\documentclass[a4paper,reqno,11pt]{amsart}
\usepackage[utf8x]{inputenc}

\usepackage{amssymb,amsfonts}
\usepackage{amsmath,amsthm}
\usepackage[all,arc, cmtip]{xy}
\usepackage{enumerate}
\usepackage{paralist}
\usepackage{booktabs} % for much better looking tables
\usepackage{array} % for better arrays (eg matrices) in maths
\usepackage{mathrsfs}
\usepackage{url}
\usepackage{color}
\usepackage{cancel}
\usepackage{todonotes} % allows notes on work in progress
\usepackage{hyperref} % hyperlinks
\usepackage{mathtools}
\usepackage[mathscr]{euscript}
\usepackage[nameinlink,capitalise,noabbrev]{cleveref}
\usepackage[shortcuts]{extdash}

\usepackage{tikz-cd}

\def\ZZ{{\mathbb{Z}}}

\def\cF{{\mathcal F}}
\def\cG{{\mathcal G}}
\def\cK{{\mathcal K}}
\def\cM{{\mathcal M}}
\def\cN{{\mathcal N}}
\def\cO{{\mathcal O}}
\def\cP{{\mathcal P}}

\newcommand{\A}{\mathcal{A}}

\newcommand{\M}{\mathcal{M}}
\newcommand{\N}{\mathcal{N}}
\newcommand{\both}{\mathrlap{\top}\bot}

\newcommand{\co}{\mathrm{co}}
\newcommand{\simp}{\otimes_{\Delta}}

\DeclareMathOperator\Aut{Aut}

\theoremstyle{definition}
\newtheorem{definition}{Definition}[section]

\newtheorem{remark}[definition]{Remark}

\theoremstyle{plain}
\newtheorem{theorem}[definition]{Theorem}
\newtheorem{lemma}[definition]{Lemma}
\newtheorem{corollary}[definition]{Corollary}
\newtheorem{proposition}[definition]{Proposition}

\newtheorem*{theorem*}{Theorem}

\AddToHook{env/theorem/begin}{\crefalias{definition}{theorem}}
\AddToHook{env/notation/begin}{\crefalias{definition}{notation}} \AddToHook{env/remark/begin}{\crefalias{definition}{remark}} \AddToHook{env/example/begin}{\crefalias{definition}{example}} \AddToHook{env/theorem/begin}{\crefalias{definition}{theorem}} \AddToHook{env/lemma/begin}{\crefalias{definition}{lemma}} \AddToHook{env/corollary/begin}{\crefalias{definition}{corollary}} \AddToHook{env/proposition/begin}{\crefalias{definition}{proposition}} \AddToHook{env/statement/begin}{\crefalias{definition}{statement}}

\title{Exterior Power Operations on Relative \nolinebreak \texorpdfstring{$K$}{K}-Theory}

\author{Bernhard K\"{o}ck}
\address{School of Mathematical Sciences\\ University of Southampton\\  Southampton SO17 1BJ\\ United Kingdom}
\email{B.Koeck@soton.ac.uk}

\author{Jane Turner}
\address{School of Mathematical Sciences\\ University of Southampton\\  Southampton SO17 1BJ\\ United Kingdom}
\email{jat1g17@soton.ac.uk}

\date{\today}

\begin{document}

\begin{abstract}
We algebraically construct exterior power operations on higher \emph{relative} algebraic $K$-groups and prove their desired properties such as the expected behaviour with respect to (tensor) products and composition. This builds on Grayson's description of relative $K$-groups in terms of explicit generators and relations and on work by Harris, the first author and Taelman for (absolute) $K$-groups. Among the new features in our approach is the observation that the product axiom in the classical notion of a $\lambda$-ring is redundant. 
\end{abstract}

\maketitle

\section*{Introduction}

Arguably most powerful and fundamental among the features of higher algebraic $K$-theory are long exact sequences of $K$-groups. The most general type of these is the long exact sequence
\begin{equation}\label{eq: long exact sequence}
\ldots \rightarrow K_{n+1}\cM \rightarrow K_{n+1}\cN \rightarrow K_n[F] \rightarrow K_n \cM \rightarrow K_n \cN \rightarrow \ldots
\end{equation}
associated with an exact functor $F \colon \cM \rightarrow \cN$ between exact categories; here, the higher relative $K$-group $K_n[F]$ can be thought of as the $n$th homotopy group of the homotopy fibre of the induced map $KF \colon K\cM \rightarrow K\cN$ between the corresponding $K$-theory spaces. Building on his earlier work in \cite{Grayson2012} for (absolute) $K$-groups using so-called binary complexes, Grayson has in \cite{Grayson2016} given a conjectural description of $K_n[F]$ in terms of generators and relations and established the sequence (\ref{eq: long exact sequence}) purely algebraically. This conjectural description of $K_n[F]$  is confirmed in \cite{Tu24} to be equivalent to the definition given above, and we use it as the definition of $K_n[F]$ in this paper. 

Exterior power operations on $K$-groups provide another important tool in the (higher) $K$-theory of schemes. 
They are the core ingredients in the theory of Adams operations and, more generally, in Grothendieck's seminal Riemann-Roch theory. Their properties, such as their behaviour with respect to (tensor) products and composition, are captured in the classical, abstract notion of a $\lambda$-ring. 

The object of this paper is to prove the following theorem, see \cref{def: lambda-ring structure on higher relative K-groups} and \cref{thm: main theorem}, which combines the two features mentioned above. 

Let $f \colon X \rightarrow Y$ be a morphism between quasi-compact schemes and let $F \colon \cM \rightarrow \cN$ denote the associated pull-back functor between the exact categories $\cM := \cP(Y)$ and $\cN:= \cP(X)$ of locally free modules of finite rank over $\cO_Y$ and $\cO_X$, respectively.

\begin{theorem*} For every $n \ge 0$, the relative $K$-group $K_n[F]$ can be equipped with products and exterior power operations $\lambda^k \colon K_n[F] \rightarrow K_n[F]$, $k \ge 1$, such that $K_n[F]$ becomes a $\lambda$-ring (without unity) and such that the incoming and outgoing maps in the long exact sequence (\ref{eq: long exact sequence}) become $\lambda$-ring homomorphisms. Moreover, if $n \ge 1$, the products on $K_n[F]$ are trivial and $\lambda^k$ is a group homomorphism. 
\end{theorem*}

This theorem will be used in a forthcoming paper to establish and prove an Adams-Riemann-Roch Theorem and a Grothen\-dieck-Riemann-Roch theorem for higher relative $K$-theory of schemes. These new theorems should be considered as refinements of more classical Riemann-Roch theorems such as the main results of \cite{Sou85}. To give more context, a similar refinement is provided by the Equivariant Tamagawa Number Conjecture for Tate motives which lifts classical results in Galois-Module Theory in Algebraic Number Theory from the projective class group to a relative $K_0$-group, see \cite{BF01} for more details.

To be able to explain the main ideas behind our theorem, we first consider the case $n=0$ and now roughly describe the generators of $K_0[F]$, see \cref{def: triples and relative K_0} for more (precise) details. They are triples of the form $X = \left(\begin{psmallmatrix} A \\ B \end{psmallmatrix}, \left(N, \begin{smallmatrix} d \\ d' \end{smallmatrix}\right), \begin{psmallmatrix} u \\ v \end{psmallmatrix}\right)$ where $A$ and~$B$ are chain complexes in $\cM$, $(N,d)$ and $(N,d')$ are chain complexes in $\cN$ (i.e., a so-called binary complex in~$\cN$) and $u\colon FA \rightarrow (N,d)$ and $v \colon FB \rightarrow (N,d')$ are quasi-isomorphisms. Using the simplicial constructions from \cite{HKT17} for (binary) complexes, we define tensor products of such triples, see \cref{def: simplicial tensor products of triples}, and exterior powers $\Lambda^k(X)$ for $k\ge 1$, see \cref{def: power operations for triples}. In order to then obtain an operation $\lambda^k \colon K_0[F] \rightarrow K_0[F]$ which is compatible with the classical operation $\lambda^k \colon K_0\cM \rightarrow K_0\cM$ via the outgoing homomorphism $K_0[F] \rightarrow K_0\cM$ in~(\ref{eq: long exact sequence}), the class $\lambda^k([X]) \in K_0[F]$ can however not be simply defined as $[\Lambda^k(X)]$, but we define it in the following a priori surprising way, see \cref{def: exterior powers on relative K_0}:
\[\lambda^k([X]) := \left[\lambda^k \left(X-\Delta\bot X\right)\right];\]
here, $\Delta\bot X$ denotes the triple $\left(\begin{psmallmatrix} B \\ B \end{psmallmatrix}, \left(N, \begin{smallmatrix} d' \\ d' \end{smallmatrix}\right), \begin{psmallmatrix} v \\ v \end{psmallmatrix}\right)$, the difference $X - \Delta \bot X$ is to be taken in the Grothendieck group $K_0B[F]$ of triples as above, the right-hand~$\lambda^k$ is computed there as well and  
$[\ldots]$ means taking the class in the factor group $K_0[F]$. A similar trick yields the definition of products on $K_0[F]$, see \cref{def: products on K_0[F]}. This way, the outgoing map $K_0[F] \rightarrow K_0\cM$ can be verified to be a $\lambda$-ring homomorphism, see \cref{prop: compatibility of products} and \cref{prop: compatibility of power operations with exact sequence}. The proof of the statement that also the incoming map $K_1\cN \rightarrow K_0[F]$ in (\ref{eq: long exact sequence}) is a $\lambda$-ring homomorphism relies on the fact (proved in \cite{HKT17}) that products in $K_1 \cN$ vanish, again see \cref{prop: compatibility of products} and \cref{prop: compatibility of power operations with exact sequence}. 

The statement that $K_0[F]$ together with the products and operations $\lambda^k$, $k\ge 1$, discussed above satisfies the properties of a $\lambda$-ring already holds on the level of $K_0B[F]$ and is proved there, see \cref{thm: K_0B[F] is a lambda ring}. More precisely, in order to show the expected behaviour of $\lambda^k$ with respect to composition, we use the approach from \cite{HKT17} which relies on the crucial result that the Grothendieck ring of the category $\mathrm{Pol}^0_{< \infty}(\ZZ)$ of polynomial functors over $\ZZ$ is isomorphic to the free $\lambda$-ring in one variable. Then the proof comes down to constructing a natural functor 
\[\mathrm{Pol}^0_{< \infty}(\ZZ) \rightarrow \mathrm{End}(B[F]).\]
Various ideas suggest themselves to be used to prove the expected behaviour of $\lambda^k$ with respect to products in our context, see the paragraph before \cref{prop: omitting product axiom}. However, we just prove the following surprising and comparatively elementary fact which seems to have not yet been observed in the literature and which will probably be convenient not just for this paper: the product axiom in the definition of a $\lambda$-ring is redundant, see \cref{prop: omitting product axiom}. 

The case when $n\ge 1$ will be treated in \cref{sec: higher relative K-groups}. This case is combinatorially more involved. The essential additional ingredient in the construction of $\lambda^k$ and in the proof of our main theorem is however just the recursive nature of the combinatorial definition of $K_n[F]$, see \cref{def: higher relative K-groups}, together with its subtle feature that the recursion can be done in any order, see \cref{def: lambda-ring structure on higher relative K-groups} and the proof of \cref{thm: main theorem}. For example, the statements that products on $K_n[F]$ are trivial and that $\lambda^k \colon K_n[F] \rightarrow K_n[F]$ is a homomorphism follow from the analogous statements for absolute $K$-groups proved in \cite{HKT17}. 

We finally recall that, in \cite{Bas68}, following \cite{Hel65}, Bass has also purely algebraically defined a $K_1$-group and a relative $K_0$-group, which we denote by $K_1^\mathrm{B}\cN$ and $K_0^\mathrm{B}[F]$, respectively, and has established the corresponding exact sequence (\ref{eq: long exact sequence}) for $n=0$ when $\cM$ and $\cN$ are split exact and the functor $F$ is cofinal. In \cref{thm:BassGraysonIso}, we provide an explicit and well-defined homomorphism 
\[\Phi \colon K_0^\mathrm{B}[F] \rightarrow K_0[F]\] 
and prove, under the same assumptions, that $\Phi$ is an isomorphism. Furthermore, in various remarks (see \ref{rem: products on Bass's K_1}, \ref{rem: products on Bass's relative K_0}, \ref{rem: power operations on Bass K_1} and \ref{rem: power operations on Bass relative K_1}), we recall/give constructions of products and exterior power operations on $K_1^\mathrm{B}\cN$ and $K_0^\mathrm{B}[F]$ and relate them to the constructions described above. 

{\bf Acknowledgements.} This work is part of the second author's PhD thesis \cite{Tu26}, supervised by the first author. The authors would like to thank the examiners Emanuele Dotto and Daniel Kasprowski for their careful reading and numerous comments, which also helped to improve this paper. They also thank the anonymous referees for their helpful comments and suggestions, in particular for a question that led to a substantial reformulation of the proof of \cref{prop: omitting product axiom}.

\section{Combinatorial Relative \texorpdfstring{$K_0$}{K0}}\label{Sec: 1}

In the first part of this section we recall Grayson's combinatorial definitions of the $K_1$-group of an exact category \cite{Grayson2012} and of the relative $K_0$-group of an exact functor between exact categories \cite{Grayson2016}. We then show that Grayson's relative $K_0$-group is isomorphic to Bass's relative $K_0$-group assuming the usual assumptions for Bass $K$-theory are satisfied. We finish by proving the expected property that swapping top and bottom in the components of the generating triples of Grayson's relative $K_0$-group amounts to a sign change. 

Let $\cM$ be an exact category. We assume throughout that all exact categories support long exact sequences in the sense of \cite[Definition~1.4]{Grayson2012}. This ensures we have a well-behaved notion of quasi\-/isomorphisms for complexes, see \cite[Definition~2.6]{Grayson2012} and the discussion after \cite[Definition~1.4]{Grayson2012}. For example, by \cite[A.9.2]{ThomTro}, every idempotent complete category supports long exact sequences.

Slightly deviating from standard notation, let $C\cM$ denote the category of bounded chain complexes in $\cM$ supported in non-negative degrees. (The latter condition could be dropped in this first section but will be essential in the later sections. Anyway, this condition does not affect the associated $K$-theory as shown in \cite[Proposition~A.1]{KZ25}.) Let $C^\mathrm{q} \cM$ denote the full subcategory of $C\cM$ consisting of acyclic complexes. Both $C\cM$ and $C^\mathrm{q}\cM$ are exact categories in the obvious way, see \cite[Section~2]{Grayson2012}.

We recall from \cite[Section 3]{Grayson2012} that a {\em binary complex} $M=\left(M., \substack{d \\ \tilde{d}}\right)$ in~$\cM$ is a graded object~$M.$ in $\cM$ together with two degree~$-1$ maps $d, \tilde{d}\colon M. \rightarrow M.$ such that both $d^2=0$ and $\tilde{d}^2=0$. We will write $\top M$ and $\bot M$ for the complexes $(M., d)$ and $(M.,\tilde{d})$, respectively, and $\both M$ for the pair $(\top M,\bot M)$.  If $d=\tilde{d}$,  the binary complex~$M$ is said to be {\em diagonal}. Given a complex $M= (M., d)$ in $C\cM$, we write $\Delta M$ for the diagonal binary complex $\left(M., \substack{ d \\d}\right)$. A {\em morphism between binary complexes} is a degree~0 map between the underlying graded objects that is a chain map with respect to both differentials. Similarly to above, let $B\cM$ denote the category of bounded binary chain complexes in $\cM$ supported in non-negative degrees. The obvious definition of short exact sequences turns $B\cM$ into an exact category. The exact full subcategory consisting of binary complexes~$M$ in~$B\cM$ such that both $\top M$ and $\bot M$ belong to $C^\mathrm{q}\cM$  will be denoted $B^\mathrm{q}\cM$.

When applied to $n=1$, Corollary~7.4 in \cite{Grayson2012}, the main outcome of \cite{Grayson2012}, shows that we have a canonical isomorphism
\begin{equation}
K_1 \cM \cong \mathrm{coker}(\Delta \colon K_0(C^\mathrm{q} \cM) \rightarrow K_0(B^\mathrm{q} \cM)).
\end{equation}
Throughout this paper, we will take this as the definition of $K_1 \cM$.

\medskip

Let $F: \M \to \N$ be an exact functor between exact categories. We recall the following crucial definitions from \cite[Section~1]{Grayson2016}.

\pagebreak[3]

\begin{definition}\label{def: triples and relative K_0}\mbox{}
\begin{enumerate}[(a)]
\item Let $B[F]$ denote the exact category whose objects are triples of the form $(M, N, u)$, where $M$ is an object in $C \M^2$ (i.e., a pair of objects in $C \M$), $N$ is an object in $B \N$ and $u: FM \to \both N$ is a pair of quasi-isomorphisms of chain complexes in~$\N$. A morphism \[ (M, N, u) \to (M', N', u') \] in~$B[F]$ is a pair of morphisms $\phi : M \to M'$ in $C\cM^2$ and $\psi : N \to N'$ in~$C\cN^2$ such that the following diagram commutes:
    \begin{equation*}
        \begin{tikzcd}
            FM \arrow[r, "F \phi"] \arrow[d, "u"]
            & FM' \arrow[d, "u'"]\\
            \both N \arrow[r, "\both \psi"]
            & \both N'.
        \end{tikzcd}
    \end{equation*}
    The subcategory $pB[F]$ is the category containing every object of $B[F]$, and whose morphisms $(\phi, \psi)$ are such that $\phi$ is a quasi-isomorphism and $\psi$ is an isomorphism.
\item Using triples $(M,N,u)$ where both $M$ and $N$ are {\it ordinary} complexes, more precisely where $M \in C\cM$,  $N \in C\cN$ and $u \colon FM \rightarrow N$ is a quasi-isomorphism, we similarly define the categories $C[F] \supset pC[F]$.
\item We extend the functors $\top, \bot$ and $\Delta$ above to triples as in (a) and~(b) by applying them componentwise (where $\top M$ and $\bot M$ for $M \in C\cM^2$ are defined similarly as for $M \in B\cM$). We also write $\tau$ for the endofunctor of $B[F]$ which swaps $\top X$ and $\bot X$ for $X \in B[F]$. 
\item The {\it relative $K_0$-group of $F$} is
\begin{equation}
K_0[F] := \mathrm{coker}(\Delta \colon K_0pC[F] \rightarrow K_0pB[F])
\end{equation}
where, by abuse of notation, $K_0pC[F]$ (and similarly $K_0pB[F]$) is defined by dividing out the equivalence relation $\sim$ on $K_0C[F]$ which is generated by $[(M,N,u)] \sim [(M',N',u')]$ whenever there is a morphism $[(M,N,u)] \rightarrow [(M',N',u')]$ in $pC[F]$.
\end{enumerate}
\end{definition}

\begin{remark}\label{rem: Relative K group justification}
The following two facts justify calling $K_0[F]$ the relative $K_0$-group.
\begin{enumerate}[(a)]
\item By \cite[Corollary~1.9]{Grayson2016}, there is an exact sequence
    \begin{equation} \label{equ: long exact sequence}
        \begin{tikzcd}
            K_1 \M \arrow[r]
            & K_1 \N \arrow[r]
            & K_0[F] \arrow[r]
            & K_0 \M \arrow[r]
            & K_0 \N
        \end{tikzcd}
    \end{equation}
where $K_1 \cM \rightarrow K_1\cN$ and $K_0 \cM \rightarrow K_0\cN$ are induced by $F$, the homomorphism $K_1 \N \to K_0[F]$ is induced by the fully faithful inclusion functor
    \[ B^\mathrm{q} \N \hookrightarrow B[F], \quad N \mapsto (0, N, 0), \]
    and $K_0[F] \to K_0 \M$ is given by
    \[ [(M, N, u )]\mapsto [\top M] - [\bot M] \]
after identifying $K_0\cM$ with $K_0qC\cM$, the Grothendieck group of $C\cM$ modulo the relations given by quasi-isomorphisms.
\item By \cite[Theorem~2.1]{Tu24}, there is a canonical isomorphism
\[K_0[F] \cong \pi_0\mathrm{hofib}(KF \colon K\cM \rightarrow K\cN)\]
between $K_0[F]$ and $\pi_0$ of the homotopy fibre of the induced map $KF$ between the $K$-theory spaces $K\cM$ and  $K\cN$ of $\cM$ and $\cN$.
\end{enumerate}
\end{remark}

Over the next pages, we'll relate $K_0[F]$ to Bass's relative $K_0$-group $K_0^\mathrm{B}[F]$ which goes back to Heller's work in \cite{Hel65}. This will illustrate how to work with the somewhat intricate definition of $K_0[F]$. Also, this relation and the later Remarks \ref{rem: products on Bass's K_1}, \ref{rem: products on Bass's relative K_0}, \ref{rem: power operations on Bass K_1} and \ref{rem: power operations on Bass relative K_1} may be useful to readers working with the more classical group $K_0^\mathrm{B}[F]$. Ultimately, we prove that $K_0[F]$ is isomorphic to $K_0^\mathrm{B}[F]$ in the following standard situation for Bass $K$-theory: $\cM$ and $\cN$ are split exact categories (i.e., all short exact sequences in $\cM$ and $\cN$ split) and $F \colon \cM \to \cN$ is a cofinal exact functor, i.e, for every object $N$ in~$\N$, there are objects $N'$ in~$\N$ and $M$ in $\M$ such that $N \oplus N' \cong F M$.

We start by recalling the definition of Bass's $K_1$-group $K_1^\mathrm{B}\cM$ from \cite[Chapter~VII, \S 1]{Bas68}. Let $\Aut(\M)$ denote the exact category of pairs $(P, \alpha)$, where $P$ is an object of $\M$ and $\alpha \colon P \to P$ is an automorphism. A morphism $(P, \alpha) \to (P', \alpha')$ in $\Aut(\M)$ is a morphism $\phi: P \to P'$ such that $\phi \circ \alpha = \alpha' \circ \phi$.  {\it Bass's $K_1$-group} $K_1^{\mathrm{B}}(\M)$ is then defined as the quotient of the Grothendieck group $K_0 \Aut(\M)$ by the subgroup generated by elements of the form
    \[ [P, \beta \circ \alpha] - [P, \beta] - [P, \alpha] \]
whenever there is an object $P$ of $\M$ with automorphisms $\alpha, \beta: P \to P$. Let $(P,\alpha, \mathrm{id})$ denote the binary complex concentrated in degrees 0 and 1 whose object in degrees 0 and 1 is $P$, whose top bottom differential is $\alpha$ and whose bottom differential is $\mathrm{id}_P$; let's call such a binary complex a {\it binary automorphism}. It is shown in \cite[Theorem~2.24]{HarPhD} that mapping an object $(P, \alpha) \in \Aut(\cM)$ to the binary automorphism $(P,\alpha, \mathrm{id})$ induces a well-defined isomorphism
\begin{equation}\label{equ: iso between Bass and Grayson K1}
K_1^\mathrm{B} \cM \xlongrightarrow{\sim} K_1 \cM
\end{equation}
if $\cM$ is split exact. 

Next, we recall the definition of Bass's relative $K_0$-group (see  \cite[Chapter~VII, \S 5]{Bas68}). Let $\co(F)$ denote the exact category  whose objects are triples of the form $(P, \alpha, Q)$, where $P, Q \in \M$ and $\alpha: FP \to FQ$ is an isomorphism in~$\N$. Furthermore, a morphism $(P, \alpha, Q) \to (P', \alpha', Q')$ consists of two morphisms $f\colon P \to P'$ and $ g\colon Q \to Q'$ in $\M$ such that $Fg \circ \alpha = \alpha' \circ Ff$. {\it Bass's relative $K_0$-group}~$K^\mathrm{B}_0[F]$ is then defined as the quotient of the Grothendieck group $K_0\mathrm{co}(F)$ by the subgroup generated by elements of the form
    \[ [P, \beta \alpha, R] - [P, \alpha, Q] - [Q, \beta, R], \]
whenever there are objects $P, Q, R$ in $\M$ with isomorphisms $\alpha: FP \to FQ$ and $\beta: FQ \to FR$.

We define the functor
\begin{equation}\label{def: functor Psi}
\Psi \colon \co(F) \to B[F], \quad (P, \alpha, Q) \mapsto \big( \big( \substack{P \\ Q} \big), FQ, \big( \substack{\alpha \\ 1} \big) \big), 
\end{equation}
where $(\substack{P \\ Q})$ is considered as a pair of complexes in $\M$ supported in degree $0$, and $FQ$ is considered as a binary complex in $\N$ supported in degree $0$; the functor $\Psi$ acts on morphisms in the obvious way.

\begin{theorem}\label{thm:BassGraysonIso}
The functor $\Psi\colon \co(F) \to B[F]$ induces a well-defined homomorphism
\[\Phi \colon K_0^{\mathrm{B}}[F] \rightarrow K_0[F].\]
If $\cM$ and $\cN$ are split exact and $F$ is cofinal, then $\Phi$ is an isomorphism. 
\end{theorem}

In the following computations we will be writing out objects of $B[F]$ supported in degrees $0$ and $1$ as follows:
\[ \big( ( A_1 \xlongrightarrow{\alpha}  A_0,  B_1 \xrightarrow{\beta} B_0  ), \big( \begin{tikzcd} N_1 \arrow[r, shift left, "d"] \arrow[r, shift right, "d'"'] & N_0 \end{tikzcd} \big), ( (u_1; u_0), (v_1; v_0) ) \big) \]
where we write the quasi-isomorphisms in an order that agrees with the ordering of the objects in the complexes.

Before we begin with the proof of \cref{thm:BassGraysonIso} we prove the following useful lemma.

\begin{lemma}\label{lem:BassShuffle}
    Given $(P, \alpha, Q) \in \co(F)$, we have the equality
    \[ [\Psi(P, \alpha, Q)] = \big[ \big( 0 \xlongrightarrow{0} 0, P \xlongrightarrow{0} Q\big), 
     \big(\begin{tikzcd} FP \arrow[r, "\alpha", shift left] \arrow[r, "0"', shift right] & FQ \end{tikzcd} \big), \big( (0;0), (1;1) \big) \big] \]
    in $K_0[F]$.
\end{lemma}

\begin{proof}
    We have the following short exact sequence in $B[F]$ obtained by na\"ively filtrating the middle triple:
    \begin{equation*}
        \begin{tikzcd}
            ((0 \arrow[r, "0"] \arrow[d, "0", shift left=1.4, rightarrowtail]
            & P, 0 \arrow[r, "0"] \arrow[d, "0", shift left=2.4, rightarrowtail] \arrow[d, "1"', shift right=1.6, rightarrowtail]
            & Q), (\, 0 \, \, \, \arrow[r, "0", shift left] \arrow[r, "0"', shift right] \arrow[d, "0", shift left=3.2, rightarrowtail] \arrow[d, "1", shift right=5, rightarrowtail]
            & FQ),((0;\alpha), (0;1))) \arrow[d, "1", shift right = 15, rightarrowtail]\\
            ((P \arrow[r, "1"] \arrow[d, "1", shift left=1.4, twoheadrightarrow]
            & P, P \arrow[r, "0"] \arrow[d, "1", shift left=2.4, twoheadrightarrow] \arrow[d, "0"', shift right=1.6, twoheadrightarrow]
            & Q), (FP \arrow[r, "\alpha", shift left] \arrow[r, "0"', shift right] \arrow[d, "1", shift left=3.2, twoheadrightarrow] \arrow[d, "0", shift right=5, twoheadrightarrow]
            & FQ),((1;\alpha),(1;1))) \arrow[d, "0", shift right=15, twoheadrightarrow]\\
            ((P \arrow[r, "0"]
            & 0, P \arrow[r, "0"]
            & 0), (FP \arrow[r, "0", shift left] \arrow[r, "0"', shift right]
            & 0),((1;0),(1;0)))
        \end{tikzcd}
    \end{equation*}
    One easily verifies that these maps are indeed morphisms in $B[F]$. Since the bottom triple is in the image of $\Delta$, we have that $[\Psi(P, \alpha, Q)]$ is equal to the middle triple in $K_0 [F]$. To show the required identity, notice that the following arrow is in $pB[F]$:
    \begin{equation*}
        \begin{tikzcd}
            ((0 \arrow[r, "0"] \arrow[d, "0", shift left=1.4]
            & 0, P \arrow[r, "0"] \arrow[d, shift left=2.4, "1"] \arrow[d, shift right=1.6, "0"']
            & Q), (FP \arrow[r, shift left, "\alpha"] \arrow[r, shift right, "0"'] \arrow[d, shift left=3.2, "1"] \arrow[d, shift right=5, "1"]
            & FQ), ((0;0),(1;1))) \arrow[d, "1", shift right=15]\\
            ((P \arrow[r, "1"]
            & P, P \arrow[r, "0"]
            & Q), (FP \arrow[r, shift left, "\alpha"] \arrow[r, shift right, "0"']
            & FQ),((1;\alpha),(1;1))).
        \end{tikzcd}
    \end{equation*}
\end{proof}

\begin{proof}[Proof (of \cref{thm:BassGraysonIso})]
As the functor $\Psi$ is obviously exact, it induces a homomorphism $ K_0 \co(F) \to K_0B[F]$. So, in order to show that $\Psi$ induces a well-defined homomorphism $\Phi \colon K_0^\mathrm{B}[F] \to K_0[F]$ it remains to check for $P, Q, R \in \cM$ and isomorphisms $\alpha \colon FP \to FQ, \beta \colon FQ \to FR$ in $\cN$ that
\begin{equation}\label{equ: additivity of Psi}
[\Psi(P, \beta \circ \alpha, R)] = [\Psi(P, \alpha, Q)] + [\Psi(Q, \beta, R)]
\end{equation}
    in $K_0[F]$. To see this we first apply \cref{lem:BassShuffle} to the object $(Q, \beta, R)$ and notice that the following morphism is in $p B[F]$:
    \begin{equation*}
        \begin{tikzcd}
            ((0 \arrow[r, "0"] \arrow[d, "0", shift left=1.4] 
            & 0, Q \arrow[r, "0"] \arrow[d, shift left = 2.4, "1"] \arrow[d, shift right=1.6, "0"']
            & R), (FQ \arrow[r, shift left, "\beta"] \arrow[r, shift right, "0"'] \arrow[d, shift left=3.2, "1"] \arrow[d, shift right=5, "1"]
            & FR),((0;0),(1;1))) \arrow[d, shift right =17, "1"]\, \, \, \, \, \, \, \, \, \, \, \, \, \\
            ((P \arrow[r, "1"]
            & P, Q \arrow[r, "0"]
            & R), (FQ \arrow[r, shift left, "\beta"] \arrow[r, shift right, "0"']
            & FR),((\alpha; \beta \circ \alpha),(1;1))).
        \end{tikzcd}
    \end{equation*}
    Now we naïvely filtrate the bottom triple above and obtain the following short exact sequence in $B[F]$:
    \begin{equation*}
        \begin{tikzcd}
            ((0 \arrow[r, "0"] \arrow[d, "0", shift left=1.4, rightarrowtail]
            & P, 0 \arrow[r, "0"] \arrow[d, "0", shift left=2.4, rightarrowtail] \arrow[d, "1"', shift right=1.6, rightarrowtail]
            & R), (0 \arrow[r, "0", shift left] \arrow[r, "0"', shift right] \arrow[d, "0", shift left=3.2, rightarrowtail] \arrow[d, "1", shift right=5, rightarrowtail]
            & FR),((0;\beta \circ \alpha), (0;1))) \arrow[d, "1", shift right = 18, rightarrowtail]\\
            ((P \arrow[r, "1"] \arrow[d, "1", shift left=1.4, twoheadrightarrow]
            & P, Q \arrow[r, "0"] \arrow[d, "1", shift left=2.4, twoheadrightarrow] \arrow[d, "0"', shift right=1.6, twoheadrightarrow]
            & R), (FQ \arrow[r, "\beta", shift left] \arrow[r, "0"', shift right] \arrow[d, "1", shift left=3.2, twoheadrightarrow] \arrow[d, "0", shift right=5, twoheadrightarrow]
            & FR),((\alpha; \beta \circ \alpha),(1;1))) \arrow[d, "0", shift right=18, twoheadrightarrow]\\
            ((P \arrow[r, "0"]
            & 0, Q \arrow[r, "0"]
            & 0), (FQ \arrow[r, "0", shift left] \arrow[r, "0"', shift right]
            & 0),((\alpha;0),(1;0)))\, \, \, \, \, \, \, \, \, \,
        \end{tikzcd}
    \end{equation*}
    We therefore have the following equalities in $K_0 [F]$:
    \begin{align*}
        [\Psi(Q, \beta, R)] &= [\Psi(P, \beta \circ \alpha, R)] + [\Psi(P, \alpha, Q)[1]]\\
        &= [\Psi(P, \beta \circ \alpha, R)] - [\Psi(P, \alpha, Q)]
    \end{align*}
    where $[1]$ denotes the obvious shift functor and the second equality follows from \cite[Lemma~6.1]{Grayson2016}. Thus \cref{equ: additivity of Psi} holds and the homomorphism $\Phi$ is defined. \\
We now assume that $\cM$ and $\cN$ are split exact and that $F$ is cofinal and show that $\Phi$ is an isomorphism. To this end, we consider the diagram
    \begin{equation*}
        \begin{tikzcd}
            K_1^{\mathrm{B}}\mathcal{M} \arrow[r] \arrow[d, "\cong"]
            & K_1^{\mathrm{B}}\mathcal{N} \arrow[r] \arrow[d, "\cong"]
            & K_0^{\mathrm{B}}[F] \arrow[r] \arrow[d, "\Phi"]
            & K_0\mathcal{M} \arrow[r] \arrow[d, equal]
            & K_0\mathcal{N} \arrow[d, equal]\\
            K_1\mathcal{M} \arrow[r]
            & K_1\mathcal{N} \arrow[r]
            & K_0 [F] \arrow[r]
            & K_0\mathcal{M} \arrow[r]
            & K_0\mathcal{N}
        \end{tikzcd}
    \end{equation*}
which is defined as follows. The bottom row is the exact sequence (\ref{equ: long exact sequence}). The top row is the exact sequence introduced in \cite[Theorem~VII.5.3]{Bas68}.
    The two left-hand vertical isomorphisms are given by (\ref{equ: iso between Bass and Grayson K1}). In order to show that $\Phi$ is an isomorphism, by the five lemma, it suffices to show the two squares involving $\Phi$ commute. We recall the homomorphism $K_0^{\mathrm{B}}[F] \to K_0 \M$ is given by $[P, \alpha, Q] \mapsto [P] - [Q]$, so the third square commutes by definition. Furthermore, the homomorphism $K_1^{\mathrm{B}} \N \to K_0^{\mathrm{B}}[F]$ is defined as follows: given an object $(N, \alpha)$ of $\Aut \N$, choose objects $N'$ in $\N$ and $M$ in $\M$ such that $N \oplus N' \cong FM$, so  that
    \[ [N, \alpha] = [N \oplus N', \alpha \oplus 1] = [FM, \widetilde{\alpha}], \]
    in $K_1^{\mathrm{B}} \N$, where $\widetilde{\alpha}$ is the automorphism of $FM$ induced by $\alpha \oplus 1$; the image of this element in $K_0^{\mathrm{B}}[F]$ is then $[M, \widetilde{\alpha}, M]$. The homomorphism $\Phi$ maps this element to $[(\substack{M \\ M}), FM, (\substack{\widetilde{\alpha} \\ 1})]$ in $K_0 [F]$. We now note we have the following exact sequence in $B[F]$ arising from na\"{i}vely filtrating the middle object:
    \begin{equation*}
        \begin{tikzcd}
            ((0 \arrow[r, "0"] \arrow[d, "0", shift left=1.4, rightarrowtail]
            & M, 0 \arrow[r, "0"] \arrow[d, "0", shift left=2.4, rightarrowtail] \arrow[d, "1"', shift right=1.6, rightarrowtail]
            & M), (0 \arrow[r, "0", shift left] \arrow[r, "0"', shift right] \arrow[d, "0", shift left=3.2, rightarrowtail] \arrow[d, "1", shift right=5, rightarrowtail]
            & FM),((0;\widetilde{\alpha}), (0;1))) \arrow[d, "1", shift right = 18, rightarrowtail]\\
            ((M \arrow[r, "1"] \arrow[d, "1", shift left=1.4, twoheadrightarrow]
            & M, M \arrow[r, "1"] \arrow[d, "1", shift left=2.4, twoheadrightarrow] \arrow[d, "0"', shift right=1.6, twoheadrightarrow]
            & M), (FM \arrow[r, "\widetilde{\alpha}", shift left] \arrow[r, "1"', shift right] \arrow[d, "1", shift left=3.2, twoheadrightarrow] \arrow[d, "0", shift right=5, twoheadrightarrow]
            & FM),((1; \widetilde{\alpha}),(1;1))) \arrow[d, "0", shift right=18, twoheadrightarrow]\\
            ((M \arrow[r, "0"]
            & 0, M \arrow[r, "0"]
            & 0), (FM \arrow[r, "0", shift left] \arrow[r, "0"', shift right]
            & 0),((1;0),(1;0))).\, \, \, \, \, \, \, \, \, \,
        \end{tikzcd}
    \end{equation*}
    Since the bottom row is diagonal the first two rows are equal in $K_0 [F]$, and since the first component of the middle row is acyclic the obvious map to the middle row from the triple
    \[ ((0 \longrightarrow  0 ,  0 \longrightarrow 0 ), (\begin{tikzcd} FM \arrow[r, shift left, "\widetilde{\alpha}"] \arrow[r, shift right, "1"'] & FM \end{tikzcd}), ((0; 0), (0; 0))) \]
    is in $pB[F]$, and thus the class of the latter triple in $K_0 [F]$ is equal to the element corresponding to the top row of the above diagram, i.e., to $[(\substack{M \\ M}), FM, (\substack{\widetilde{\alpha} \\ 1})]$. On the other hand, this element is the image of the element $[FM, \widetilde{\alpha}]$ when mapped first via the isomorphism $K_1^{\mathrm{B}}(\N) \to K_1(\N)$ given in (\ref{equ: iso between Bass and Grayson K1}) and then by the homomorphism $K_1(\N) \to K_0 [F]$ as in (\ref{equ: long exact sequence}). Hence the second square in the diagram above commutes as well, and the proof of \cref{thm:BassGraysonIso} is complete.
\end{proof}

\begin{remark}
As an alternative to the latter part of the proof above, the reader may wish to check that $\Phi$ is compatible with with the isomorphisms from $K_0^{\text{B}}[F]$ and $K_0[F]$ to $ \pi_0\mathrm{hofib}(KF \colon K\cM \rightarrow K\cN)$ established in \cite[Theorem~1.5]{Iwa19} and \cite[Theorem~2.1]{Tu24}, respectively. 
\end{remark}

We end this section with the following useful fact. By \cite[Lemma~6.1]{Grayson2016} we know that shifting the complexes in a triple  $X \in B[F]$ by 1 changes the sign of its class $[X]$ in $K_0[F]$. The following proposition shows that the same holds when swapping top and bottom in $X$; this is similar to \cite[Corollary~8.7]{Grayson2016} where this is proved for $K_1 \cN$.

\begin{proposition}\label{prop: sign change for swapping}
For $X \in B[F]$ we have $[\tau X] = - [X]$ in $K_0[F]$. 
\end{proposition}

\begin{proof}
We write $X = \left(\begin{psmallmatrix} A \\ B \end{psmallmatrix}, \left(N, \begin{smallmatrix} d \\ d' \end{smallmatrix}\right), \begin{psmallmatrix} u \\ v \end{psmallmatrix}\right)$. Let $s$ denote the endomorphism of $N \oplus N$ which swaps the two summands and consider the two complex homomorphisms $u \oplus v$ and $s \circ (u\oplus v)$ from $F(A) \oplus F(B)$ to $\top N \oplus \bot N$ and $\bot N \oplus \top N$, respectively. The triple $X \oplus \tau X$ is isomorphic to the triple $\left(\Delta(A \oplus B), \left(N\oplus N, \begin{smallmatrix} d \oplus d' \\ d' \oplus d \end{smallmatrix}\right), \left(\begin{smallmatrix} u \oplus v \\ s \circ (u \oplus v) \end{smallmatrix}\right) \right)$. By the second half of the proof of \cite[Lemma~6.3]{Grayson2016}, its class $[X \oplus \tau X]$ in~$K_0[F]$ is therefore equal to the image of the class $\left[\mathrm{Cone} \binom{u \oplus v}{s \circ (u\oplus v)}  \right]\in K_1 \cN$ under the map $K_1 \cN  \rightarrow  K_0[F]$ from (\ref{equ: long exact sequence}); here,  $\mathrm{Cone} \binom{u \oplus v}{ s \circ (u \oplus v)}$ is the binary complex with graded object $FA[1] \oplus FB[1] \oplus N \oplus N$ and which shares its differentials with $\mathrm{cone}(u \oplus v)$ and $\mathrm{cone} (s \circ (u \oplus v))$, respectively. Generalising the proof of \cite[Lemma~3.2]{KKW}, we will show in the next paragraph that this element lies in the image of the map $K_1\cM \rightarrow K_1\cN$ induced by~$F$; then, by the exactness of (\ref{equ: long exact sequence}), the class $[X \oplus \tau X]$ in $K_0[F]$ vanishes and hence $[\tau X] = - [X]$, as claimed. \\
By definition we have the binary ladder
\[\left(\mathrm{Cone} \binom{u \oplus v}{ s \circ (u \oplus v)}, \mathrm{Cone} \binom{u \oplus v}{u \oplus v}, 1_{FA\oplus FB} \oplus 1_{\top N \oplus \bot N}, 1_{FA\oplus FB} \oplus s\right)\]
in the sense of \cite[Definition 2.2(a)]{KKW}, i.e., the third component of this quadruple is an isomorphism between the top parts of the binary complexes in the first two components and the fourth component is an isomorphism between the bottom parts.  
\cite[Lemma~3.3]{KKW} then yields the second in the following chain of equalities in $K_1\cN$: 
 \[ \left[\mathrm{Cone} \binom{u \oplus v}{s \circ (u\oplus v)}  \right] = \left[\mathrm{Cone} \binom{u \oplus v}{s \circ (u\oplus v)}  \right] - \left[\mathrm{Cone} \binom{u \oplus v}{u\oplus v}  \right] \]
 \[=\sum_{i=0}^\infty (-1)^i \left[ \hspace*{-0.5em} \begin{tikzcd} FA_{i-1} \oplus FB_{i-1} \oplus N_i \oplus N_i \arrow[r, shift left, "1 \oplus 1"] \arrow[r, shift right, "1 \oplus s"'] & FA_{i-1} \oplus FB_{i-1} \oplus N_i \oplus N_i \end{tikzcd} \hspace*{-0.5em} \right]\]
 \begin{equation}\label{eq: element}
 = \sum_{i=0}^\infty (-1)^i \left[\begin{tikzcd}  N_i \oplus N_i \arrow[r, shift left, "1"] \arrow[r, shift right, "s"'] & N_i \oplus N_i \end{tikzcd} \right];
 \end{equation}
 note that the binary complexes $\begin{tikzcd} FA_{i-1} \oplus FB_{i-1}  \arrow[r, shift left, "1"] \arrow[r, shift right, "1 "'] & FA_{i-1} \oplus FB_{i-1} \end{tikzcd}$ are diagonal, which shows the third equality above. As $u$ is a quasi-iso\-morphism, the Euler characteristic of $\mathrm{cone}(u)$ vanishes; in other words, the Euler characteristic of $N$ is the same as of $FA$. Hence the same is true for their images under the homomorphism $K_0\cN \rightarrow K_1\cN$ induced by the exact functor $P \mapsto \left(\begin{tikzcd}  P \oplus P \arrow[r, shift left, "1"] \arrow[r, shift right, "s"'] & P \oplus P \end{tikzcd}\right)$. This finally proves that the alternating sum (\ref{eq: element}) is equal to 
 \[\sum_{i=0}^\infty (-1)^i \left[\begin{tikzcd}  FA_i \oplus FA_i \arrow[r, shift left, "1"] \arrow[r, shift right, "s"'] & FA_i \oplus FA_i \end{tikzcd} \right]\] 
 which is obviously in the image of the map $K_1\cM \rightarrow K_1\cN$, as we wanted to show. 
\end{proof}

\section{Products on Relative \texorpdfstring{$K_0$}{K0}}\label{Sec: 2}

The object of this section is to turn Grayson's relative $K_0$-group into a (non-unital) commutative ring. To this end, we extend the definition of so-called simplicial tensor products of complexes \cite{HKT17} to generating triples of the relative~$K_0$-group and then use a certain linear combination of these tensor products to define products in the relative $K_0$-group so that also the outgoing map in the associated exact sequence of $K$-groups becomes a ring homomorphism. We end by briefly explaining the corresponding constructions and properties for Bass $K$-theory. 

We start by recalling the Dold-Kan correspondence. Let $\Delta$ denote the category whose objects are the totally ordered sets
    \[ [n] = \{ 0 < 1 < ... < n-1 < n \}, \quad  n \in \mathbb{N},\]
and whose morphisms are the non-decreasing functions. (We hope denoting both this category and the diagonal functor introduced in \cref{Sec: 1} by $\Delta$ won't lead to any confusion.)  Given any category~$\A$, a \emph{simplicial object in~$\A$} is a functor $\Delta^{\mathrm{op}} \to \A$. The category of such simplicial objects in $\A$ is denoted by $\A^{\Delta^{\mathrm{op}}}$. For  $X \in \A^{\Delta^{\mathrm{op}}}$, we denote the image of $[n]$ by $X_n$. If $\A$ is an idempotent complete additive category, the Dold-Kan correspondence \cite[Theorem~1.2.3.7]{HigherAlg} gives us an equivalence of categories
    \begin{equation*}
        \begin{tikzcd}
            \A^{\Delta^{\mathrm{op}}} \arrow[r, shift left, "N"]
            & C^\infty \A \arrow[l, shift left, "\Gamma"]
        \end{tikzcd}
    \end{equation*}
between $\A^{\Delta^{\mathrm{op}}}$ and the category $C^\infty \A$ of chain complexes in $\A$ supported in non-negative degrees. Here, as usual, $N$ denotes the normalized chain complex (also called Moore complex) functor.

We now start to define products on relative $K_0$-groups. So, let $\cM$ be an idempotent complete exact category equipped with a bi-exact functor $\otimes: \M \times \M \to \M$, called a \emph{tensor product}, which we assume to be associative and commutative in the usual sense. Furthermore, we assume we have a distinguished object $I$ of~$\M$ which behaves as a ``multiplicative identity" in the obvious way. We recall from \cite[Definition~5.4]{HKT17}, given two chain complexes $A$ and $B$ in $C \M$, the \emph{simplicial tensor product of $A$ and $B$} is defined as follows:
    \[ A \simp B := N( \mathrm{diag}( \Gamma A \otimes \Gamma B)) \]
where $\mathrm{diag}(\Gamma A \otimes \Gamma B) \in \cM^{\Delta^{\mathrm{op}}}$ is given by  $\mathrm{diag}( \Gamma A \otimes \Gamma B)_n:= (\Gamma A)_n \otimes (\Gamma B)_n$ (i.e., the diagonal of the bi-simplicial object $\Gamma A \otimes \Gamma B$). Note that, by \cite[Lemma 5.6]{HKT17}, $A \simp B$ is bounded again, i.e., an object of $C\cM$. We extend this definition to binary chain complexes by viewing them as objects in $C \cM^2$ and applying the simplicial tensor product componentwise. The resulting pair of chain complexes is then a binary chain complex again.

Note, by the Eilenberg-Zilber Theorem \cite[Satz 2.15]{DP61}, the simplicial tensor product $A \otimes_\Delta B$ is homotopy equivalent to the total complex of the double complex $A \otimes B$. We choose to work with the former product because the former is more directly compatible with the exterior powers used in the next section. 

We now assume we are given two idempotent complete exact categories $\M$ and $\N$ with tensor products and with identities $I_{\M}$ and $I_{\N}$, respectively, and let  $F: \M \to \N$ be an exact functor which respects tensor products in the obvious way, and such that $FI_{\M}=I_{\N}$.

\begin{definition}\label{def: simplicial tensor products of triples}
    Given two objects $X = (M, N, u)$ and $Y = (M', N', u')$ in~$C[F]$ or in $B[F]$, we define their \emph{simplicial tensor product} as follows:
    \[ X \simp Y := (M \simp M', N \simp N', u \simp u'). \]
\end{definition}

Note that, by the Eilenberg-Zilber Theorem \cite[Satz~2.15]{DP61}, the simplicial tensor product $u \simp u'$ is a quasi-isomorphism again.

\begin{lemma}
    The simplicial tensor products on $C[F]$ and $B[F]$ induce well-defined products on $K_0 C[F]$, $K_0 B[F]$, $K_0 p C[F]$ and $K_0 p B[F]$ which turn these Grothendieck groups into commutative rings with unity.
\end{lemma}

\begin{proof}
    It is obvious that the simplicial tensor product induces well-defined products on $K_0 C[F]$ and $K_0 B[F]$ and that the class of the element $[I_{\M}, I_{\N}, 1]$ is an identity element, where $I_{\M}$ and $I_{\N}$ are considered as (ordinary or binary pairs of) complexes concentrated in degree zero. The Eilenberg-Zilber theorem \cite[Satz~2.15]{DP61} implies that these ring structures induce ring structures on $K_0 p C[F]$ and $K_0 p B[F]$.
\end{proof}

Note that the simplicial tensor product on $K_0pB[F]$ does not directly induce a product on the quotient $K_0[F]$ because the image of the ring homomorphism $\Delta \colon K_0pC[F] \rightarrow K_0pB[F]$ contains $1$ and is hence not an ideal. Nonetheless, we have a natural product on $K_0[F]$ which we define now. We remember that the diagonal functor $\Delta$ is split by $\bot$ (or~$\top$) and that therefore the obvious homomorphism
\begin{equation} \label{Equ: identification}
\mathrm{ker}\left(\bot \colon K_0pB[F] \to K_0pC[F]\right) \longrightarrow \mathrm{coker}\left(\Delta \colon K_0pC[F] \to K_0pB[F]\right)
\end{equation}
is bijective. We accordingly identify $K_0[F]$, i.e., the target of this homomorphism, with its source which is an ideal in $K_0pB[F]$ and which in particular inherits a well-defined product. This is the product structure on~$K_0[F]$ (without 1) we will be considering. In other words:

\begin{definition}\label{def: products on K_0[F]}\label{def: products on relative K_0}
Given $x, y \in K_0[F]$, we choose representatives $\tilde{x}, \tilde{y}$ in $K_0pB[F]$ and define the \emph{product} $xy$ as the class of $(\tilde{x} - \Delta \bot \tilde{x})(\tilde{y} - \Delta \bot \tilde{y})$ in $K_0[F]$.
\end{definition} 

If $x, y$ are the classes of triples $X, Y \in B[F]$, then, by \cref{prop: sign change for swapping}, their product $xy$ in $K_0[F]$ can be written as the class of a single triple as well: 
\[xy = [(X \otimes_\Delta Y) \oplus (\tau X \otimes_\Delta \Delta \bot Y) \oplus (\Delta \bot X \otimes_\Delta \tau Y)].\]
(Beware, although  $[\Delta \bot {X}]=0 $ in $K_0[F]$, the class of $(\Delta \bot X) \otimes_\Delta \tau Y$ does not vanish in general.)

\begin{proposition} \label{prop: compatibility of products}
We equip $K_0\cM$ with the usual product and $K_1\cN$ with the trivial product. Then these products are compatible with the just defined product on $K_0[F]$ via the exact sequence (\ref{equ: long exact sequence}).
\end{proposition}

\begin{proof}
We first show that the outgoing homomorphism $K_0[F] \rightarrow K_0\cM$ is compatible with products. Let $x, x' \in K_0[F]$ be represented by $X = (M,N,u) $ and $X' = (M',N',u')$ in $B[F]$. Then the product $xx'$ which is the class of 
\[X \otimes_\Delta X' - (\Delta \bot X) \otimes_\Delta X' - X \otimes_\Delta (\Delta\bot X') + (\Delta\bot X) \otimes_\Delta (\Delta \bot X')\]
in $K_0[F]$ is mapped to 
\begin{align*}(\top& M \otimes_\Delta \top M' - \bot M \otimes_\Delta \bot M')- (\bot M \otimes_\Delta \top M'- \bot M \otimes_\Delta \bot M')\\
& \hspace*{0.5em} - (\top M \otimes_\Delta \bot M' - \bot M \otimes_\Delta \bot M') + (\bot M \otimes_\Delta \bot M' - \bot M \otimes_\Delta \bot M')\\
&= \top M \otimes_\Delta \top M' - \bot M \otimes_\Delta \top M' - \top M \otimes_\Delta \bot M' + \bot M \otimes_\Delta \bot M' \\
&= (\top M - \bot M) \otimes_\Delta (\top M' - \bot M')
\end{align*}
which in turn is equal to the product of the images of $x$ and $y$ in $K_0\cM$; note that the product on $K_0qC\cM$ given by simplicial tensor products is compatible with the usual product on $K_0\cM$ via the canonical isomorphism $K_0\cM \rightarrow K_0qC\cM$. 
\\
In order to show that the incoming homomorphism $K_1 \cN \rightarrow K_0[F]$ is compatible with products, we recall from \cite[Proposition~5.11] {HKT17} that the class $[N \otimes_\Delta N']$ in $K_1 \cN$ of the simplicial tensor product of any two binary complexes $N,N'$ in $B^\mathrm{q} \cN$ vanishes. In particular the trivial product on $K_1\cN$ can be interpreted as the product induced by the simplicial tensor product or, even more complicatedly, by the product defined analogously to the product in \cref{def: products on K_0[F]}. The latter product is obviously compatible with the product in $K_0[F]$ via the incoming homomorphism.
\end{proof}

\begin{remark} Similarly to above, one easily shows that Proposition 2.4 is also true if we both replace 
$\bot$ with $\top$ in \cref{def: products on K_0[F]} and the outgoing homomorphism $K_0[F] \rightarrow K_0 \cM$ in (\ref{equ: long exact sequence}) with its negative. 
\end{remark}

\begin{remark} \label{rem: products on Bass's K_1}
We recall the following folklore statements for Bass's $K_1$. 
\begin{enumerate}[(a)]
\item Given two objects $(P, \alpha)$ and $(Q, \beta)$ of $\Aut(\M)$ we define their \emph{tensor product} as follows:
    \[ (P, \alpha) \otimes (Q, \beta) := (P \otimes Q, \alpha \otimes \beta) \]
    This turns $K_0 \mathrm{Aut}(\cM)$ into a commutative ring with unity.
\item Let 
\[\tilde{K}_0 \mathrm{Aut} (\cM) := \mathrm{ker}(K_0 \bot \colon K_0 \mathrm{Aut}(\cM) \rightarrow K_0 \cM)\] 
where $K_0 \bot$ is induced by the functor $\bot \colon (P,\alpha) \mapsto P$ and let $\tilde{I}$ denote the subgroup of $\tilde{K}_0 \mathrm{Aut} (\cM)$ generated by elements of the form 
\[[P, \beta \circ \alpha] - [P,\alpha] - [P, \beta] + [P,1]\] 
whenever there is an object $P$ of $\cM$ with automorphisms~$\alpha, \beta$. Similarly to the considerations before \cite[Lemma 1.2]{Koe00} we see that $\tilde{I}$ is an ideal of $K_0 \Aut(\M)$. Furthermore, we have the isomorphism
    \[ K_0 \Aut(\M) / \widetilde{I} \cong K_0 \M \oplus K_1^{\mathrm{B}}(\M). \]
    The resulting product $xy$ for $x, y \in K_0 \M \oplus K_1^{\mathrm{B}}(\M)$ is the usual one if $x, y \in K_0 \M$, is trivial if $x, y \in K_1^{\mathrm{B}}(\M)$ and defines a $K_0 \M$-module structure on $K_1^{\mathrm{B}}(\M)$ if $x \in K_0 \M$ and $y \in K_1^{\mathrm{B}}(\M)$.
\end{enumerate}
\end{remark}

\begin{remark}\label{rem: products on Bass's relative K_0}\mbox{}
The following statements concerning Bass's relative $K_0$-group are easy to prove (similarly to above, see also the considerations right before \cite[Lemma~1.3]{Koe00}):
\begin{enumerate}[(a)]
    \item Let $ X:=(P, \alpha, Q)$ and $Y:=(R, \beta, S)$ be objects in $\co(F)$. Their tensor product in $\mathrm{co(F)}$ is defined by
       \[ X \otimes Y := (P \otimes R, \alpha \otimes \beta, Q \otimes S). \]
    This turns $K_0\mathrm{co}(F)$ into a commutative ring with unity.
\item Let 
\[\tilde{K}_0 \mathrm{co}(F) := \mathrm{ker}(K_0 \bot \colon K_0 \mathrm{co}(F) \rightarrow K_0 \cM)\] 
where $K_0 \bot$ is induced by the functor $\bot \colon (P,\alpha,Q) \mapsto Q$ and let $\tilde{J}$ denote the subgroup of $\tilde{K}_0 \mathrm{co} (F)$ generated by elements of the form 
\[[P, \beta \circ \alpha,R] - [P,\alpha,Q] - [Q, \beta, R] + [Q,1,Q]\] 
whenever there are objects $P,Q,R$ in $\cM$ and isomorphisms~$\alpha \colon FP \rightarrow FQ$ and $\beta \colon FQ \rightarrow FR$ in $\cN$. Then $\tilde{J}$ is an ideal of $K_0 \mathrm{co}(F)$ and we have the isomorphism
    \begin{equation} \label{equ: isomorphism for Bass relative K_0}
    K_0 \mathrm{co}(F) / \widetilde{J} \cong K_0 \M \oplus K_0^\mathrm{B}[F]. 
    \end{equation}
This isomorphism gives rise to a multiplication on $K_0^\mathrm{B}[F]$ which, in contrast to the multiplication on $K_1^\mathrm{B}\cM$ (see \cref{rem: products on Bass's K_1}(b)), is in general not trivial. 
\item  This product on $K_0^\mathrm{B}[F]$  is compatible with the trivial product on $K_1\cN$ (see \cref{rem: products on Bass's K_1}(b)) and the usual product on $K_0\cM$ via Bass's exact sequence \cite[Theorem~VII.5.3]{Bas68} assuming $\cM$ and $\cN$ are split exact and $F$ is cofinal.
\item The functor $\Psi \colon \mathrm{co}(F) \rightarrow B[F]$ (see (\ref{def: functor Psi})) is compatible with tensor products. The induced homomorphism $\Phi \colon K_0^\mathrm{B}[F] \rightarrow  K_0[F] $ (see \cref{thm:BassGraysonIso}) is compatible with products. 
\end{enumerate}
\end{remark}

\section{Power Operations on Relative \texorpdfstring{$K_0$}{K0}}\label{Sec: 3}

In this section we equip Grayson's relative $K_0$-group with exterior power operations. To this end, similarly to the previous section, we first extend the definition of exterior powers of complexes \cite{HKT17} to generating triples of the relative~$K_0$-group and then use a combination of such exterior powers to define exterior power operations on the relative $K_0$-group which are compatible with both the incoming and outgoing map in the associated exact sequence of $K$-groups. As in the previous section, we end by briefly explaining the corresponding constructions and properties for Bass $K$-theory. 

We assume we are given an assembly of power operations as in \cite[Definition 1.1]{KZ25}. Recall this means we are given a sequence $\cM_n$, $n \ge 0$, of exact categories (in most examples, all $\cM_n$ are equal to a single exact category $\cM$), a bi-exact functor 
\[ \otimes \colon \cM_n \times \cM_p \rightarrow \cM_{n+p} \quad \textrm{ for each } n,p \ge 0,\]
and a functor
\[\textrm{Fil}_k (\cM_n) \rightarrow \cM_{nk} \quad \textrm{ for each } k \ge 1, n \ge 0,\]
which maps an object of $\textrm{Fil}_k(\cM_n)$, i.e., a sequence $V_1 \rightarrowtail \ldots \rightarrowtail V_k$ of admissible monomorphisms in $\cM_n$ of length~$k-1$, to an object $V_1 \wedge \ldots \wedge V_k$ in $\cM_{nk}$, and these data are subject to certain axioms. If all the $k-1$ monomorphisms are the identity on $V:= V_1 = \ldots =V_k$, we write $\bigwedge ^k V$ for $V_1 \wedge \ldots \wedge V_k$. A reader who prefers not to work in this axiomatic setup may just assume that the given assembly of power operations is the standard one where every $\cM_n$ is equal to the exact category $\cP(X)$ of locally free $\cO_X$-modules of finite rank on a quasi-compact scheme~$X$, the functor $\otimes$ is the usual tensor product and $V_1 \wedge \ldots \wedge V_k$ is defined to be the image of the tensor product $V_1 \otimes \ldots \otimes V_k$ in the $k^\mathrm{th}$ exterior power $\bigwedge^k V_k$. 

We furthermore assume that all categories $\cM_n$, $n \ge 0$, are idempotent complete. By \cite[Proposition 1.2]{KZ25}, we obtain an assembly of power operations on the categories $C\cM_n$, $n \ge 0$, as follows. Similarly to \cref{Sec: 2}, the simplicial tensor product defines functors 
\[ \otimes_\Delta \colon C \cM_n \times C \cM_p \rightarrow C\cM_{n+p} \quad n, p \ge 0; \]
in the same vein, mapping a sequence $V_{1}. \rightarrowtail \ldots \rightarrowtail V_k.$ in $\textrm{Fil}_k(C \cM_n)$ to 
\[V_1. \wedge \ldots \wedge V_k. := N (\Gamma V_1. \wedge \ldots \wedge \Gamma V_k.) \quad \textrm{ in } \quad  C \cM_{nk} \]
defines a functor
\[\textrm{Fil}_k(C\cM_n) \rightarrow C\cM_{nk}\quad \textrm{ for every } k \ge 1, n\ge 0.\]
Furthermore, according to \cite[Corollary~1.7]{KZ25}, by applying the construction above to each of the two differentials in binary complexes we obtain an assembly of power operations on $B\cM_n$, $n\ge 0$.

We now furthermore assume we are given two assemblies of power operations, on $\cM_n$, $n \ge 0$, and on $\cN_n$, $n \ge 0$, respectively, and exact functors $F_n \colon \cM_n \rightarrow \cN_n$, $n \ge 0$, which respect the given structures in the obvious sense. The main example we have in mind is the pull-back functor $f^*: \cP(Y) \rightarrow \cP(X)$ for a given morphism $f \colon X \rightarrow Y$ between quasi-compact schemes $X$ and $Y$. We also assume that, for every quasi-isomorphism $f.$ in~$C\cM_n$ and in  $C\cN_n$ and for all $k \ge 1$, also $\bigwedge^k(f.)$ is a quasi-isomorphism. This additional assumption is satisfied for example for the standard assembly of power operations on $\cP(X)$, see \cite[Proposition~1.9]{KZ25}.

\begin{definition}\label{def: power operations for triples}
Given $k \ge 1$ and $n\ge 0$ and a sequence $X_1 \rightarrowtail \ldots \rightarrowtail X_k$ of admissible monomorphisms in $C[F_n]$ or $B[F_n]$, where $X_i = (M_i, N_i, u_i)$, we define the triple 
\[X_1 \wedge \ldots \wedge X_k := (M_1 \wedge \ldots \wedge M_k, N_1 \wedge \ldots \wedge N_k, u_1 \wedge \ldots \wedge u_k) \]
in $C[F_{nk}]$ or $B[F_{nk}]$.
\end{definition}

Note that, by \cite[Proposition 1.4]{KZ25}, the morphism $u_1 \wedge \ldots \wedge u_k$ is a quasi-isomorphism again. 

\begin{lemma}\label{lem: power operations on B[F]}
The assignments 
\[(X,Y) \mapsto X \otimes_\Delta Y \textrm{ and } (X_1 \rightarrowtail \ldots \rightarrowtail X_k) \mapsto X_1 \wedge \ldots \wedge X_k\]
define an assembly of power operations on the exact categories $C[F_n]$, $n \ge 0$, and $B[F_n]$, $n \ge 0$, which preserve the subcategories $pC[F_n]$ and $pB[F_n]$, respectively. 
\end{lemma}

\begin{proof}
Straightforward. 
\end{proof}

In particular we obtain power operations 
\[\lambda^k \colon K_0C[F_n] \rightarrow K_0 C[F_{nk}] \quad \textrm{ and } \quad \lambda^k \colon K_0 B[F_n] \rightarrow K_0 B[F_{nk}] \] 
and 
\[\lambda^k \colon K_0p C[F_n] \rightarrow K_0p C[F_{nk}] \quad \textrm{ and } \quad \lambda^k \colon K_0p B[F_n] \rightarrow K_0p B[F_{nk}]. \]
As in \cref{Sec: 2}, we employ the identification (\ref{Equ: identification}) to define power operations $\lambda^k \colon K_0[F_n] \rightarrow K_0[F_{nk}]$. More explicitly: 

\begin{definition}\label{def: exterior powers on relative K_0}
Given $x \in K_0[F_n]$, we choose a representative $\tilde{x}$ of $x$ in $K_0pB[F_n]$ and define $\lambda^k(x)$ to be the class of $\lambda^k(\tilde{x}-\Delta \bot\tilde{x})$ in $K_0[F_{nk}]$. 
\end{definition}

If $x = [X]$ is the class of a triple $X$ in $B[F_n]$, we can write $\lambda^k(x)$ as the class of a single triple in the following way. 

\begin{lemma}\label{lem: computing lambda^k} 
Let $X = \left(\begin{psmallmatrix} A \\ B \end{psmallmatrix}, \left(N, \begin{smallmatrix} d \\ d' \end{smallmatrix}\right), \begin{psmallmatrix} u \\ v \end{psmallmatrix}\right) \in B[F_n]$. Then we have:
    \[ \lambda^k([ X]) = \left[ \left( \left( \substack{\Lambda_+^k(A, B) \\ \Lambda_-^k(A, B)} \right), \left( \widehat{\Lambda}^k N, \substack{\Lambda_+^k(d, d') \\ \Lambda_-^k(d, d')} \right), \left( \substack{\Lambda_+^k(u, v) \\ \Lambda_-^k(u, v)} \right) \right)\right] \quad \textrm{ in } \quad K_0[F_{nk}],\]
where $\Lambda_+^k, \Lambda_-^k$ and $\widehat{\Lambda}^k $ are defined as follows:
    \begin{align*}
        \Lambda_+^k(A, B) &= \bigoplus_{\substack{a+b_1+\cdots+b_u=k \\ u \text{ even} \\ a, u \geq 0, b_1, ..., b_u \geq 1}} \Lambda^a(A) \otimes_\Delta \Lambda^{b_1}(B) \otimes \cdots \otimes_\Delta \Lambda^{b_u}(B),\\
        \Lambda_-^k(A, B) &= \bigoplus_{\substack{a+b_1+\cdots+b_u=k \\ u \text{ odd} \\ a, u \geq 0, b_1, ..., b_u \geq 1}} \Lambda^a(A) \otimes_\Delta \Lambda^{b_1}(B) \otimes \cdots \otimes_\Delta \Lambda^{b_u}(B),\\
        \widehat{\Lambda}^k N & = \bigoplus_{\substack{a+b_1+\cdots+b_u=k \\ u \geq 0, a, b_1, ..., b_u \geq 1}} \Lambda^a N \otimes_\Delta \Lambda^{b_1} N \otimes_\Delta \cdots \otimes_\Delta \Lambda^{b_u} N.
    \end{align*}
\end{lemma} 

Note, in order to match up the index sets for $\Lambda^k_+(d, d')$ and $\widehat{\Lambda}^k N$ here, we consider any tuple $(a, b_1, \dots, b_u)$ in $\Lambda^k_+(d, d')$ as the tuple $(\tilde{a}, \tilde{b}_1, \ldots, \tilde{b}_{u-1})$ in $\widehat{\Lambda}^k N$ where $\tilde{a} = b_1, \tilde{b}_1 = b_2, \ldots, \tilde{b}_{u-1} = b_u$ if $a=0$, and as the tuple $(a, b_1, \dots, b_u)$ if $a \ge 1$; similarly for the index sets of $\Lambda^k_-(d, d')$ and $\widehat{\Lambda}^k N$.

\begin{proof}
Using \cite[Equation 2.3]{GraysonExterior} which computes $\lambda^k$ of a difference, we obtain: 
\begin{align*}
\lambda^k ([X]) &= \lambda^k ([X] - [\Delta\bot X])\\
& = \sum_{\substack{a+b_1+\cdots+b_u=k \\ a, u \geq 0, b_1, ..., b_u \geq 1}} (-1)^u \lambda^a([X])  \lambda^{b_1}([\Delta \bot X])  \cdots \lambda^{b_u}([\Delta \bot X])\\
& = \left[\bigoplus_{\substack{a+b_1+\cdots+b_u=k \\ u \geq 0, a, b_1, ..., b_u \geq 1}} \Lambda^a(\tau^u X) \otimes_\Delta \Lambda^{b_1}(\Delta \bot X) \otimes_\Delta \cdots \otimes_\Delta \Lambda^{b_u}(\Delta \bot X)\right]
\end{align*}
To see the latter equality, we use the fact that $\tau$ commutes with $\Lambda^a$ and amounts to a minus sign in $K_0^\mathrm{B}[F_{nk}]$ according to \cref{prop: sign change for swapping} and we notice that the terms for $a=0$ vanish in our context. Finally, to see that the term obtained is indeed equal to $\left[ \left( \left( \substack{\Lambda_+^k(A, B )\\ \Lambda_-^k(A, B)} \right), \left( \widehat{\Lambda}^k N, \substack{\Lambda_+^k(d, d') \\ \Lambda_-^k(d, d')} \right), \left( \substack{\Lambda_+^k(u, v) \\ \Lambda_-^k(u, v)} \right) \right)\right]$,  we split the index set into two halves according to whether $u$ is even or odd and then, for the top components in the triples, we interpret any index where $(a \ge 1, u \textrm{ odd})$ as an index where $(a = 0, u \textrm{ even})$, and similarly for the bottom components. 
\end{proof}

\begin{proposition}\label{prop: compatibility of power operations with exact sequence} Let $k \ge 1$. Via the exact sequence (\ref{equ: long exact sequence}), the just defined power operation $\lambda^k \colon K_0[F_n] \rightarrow K_0[F_{nk}]$ is compatible with the power operation $\lambda^k \colon K_0\cM_n \rightarrow K_0\cM_{nk}$ and $\lambda^k \colon K_1\cN_n \rightarrow K_1\cN_{nk}$ resulting from the given assemblies of power operations on $\cM_n$ and $\cN_n$. 
\end{proposition}

\begin{proof} By \cref{prop: sign change for swapping}, any $x \in K_0[F_n]$ can be represented by a triple $X \in B[F_n]$. We use the same notation for the components of $X$ as in \cref{lem: computing lambda^k}. This lemma then shows that the difference in $K_0qC\cM_{nk}$ of the classes of the top and bottom complexes in the first component of $\lambda^k([X])$ is equal to $[\Lambda_+^k(A, B)] - [\Lambda_-^k(A, B)]$ which in turn is equal to $\lambda^k([A] - [B])$, again by \cite[Equation 2.3]{GraysonExterior}; note $\lambda^k$ here denotes the power operation resulting from the assembly of power operations on $C\M_n$, $n\ge 0$ (see above) which in turn is obviously compatible with the power operation $\lambda^k \colon K_0\cM_n \rightarrow K_0\cM_{nk}$ via the canonical isomorphism $K_0\cM_{nk} \rightarrow K_0 qC\cM_{nk}$. This shows the first statement. \\
To prove the second statement, let $x \in K_1\cN_n$ be represented by $N \in B^\mathrm{q}\cN_n$. Then, by definition of $\lambda^k$ on $K_1$ (see \cite[(3.8)]{KZ25}) we have in $K_1\cN_{nk}$:
\begin{align*}
    \lambda^k(x) &= [\Lambda^k(N)] \\
    &= \sum_{\substack{a+b_1+\cdots+b_u=k \\ a, u \geq 0, b_1, ..., b_u \geq 1}} (-1)^u [\Lambda^a(N) \otimes_\Delta \Lambda^{b_1}(\Delta \bot N) \otimes_\Delta \cdots \otimes_\Delta \Lambda^{b_u}(\Delta \bot N)]
\end{align*}
where all summands for $u\not= 0$  vanish by \cite[Proposition~5.11] {HKT17}. This element is mapped to the element
\begin{align*}
\sum_{\substack{a+b_1+\cdots+b_u=k \\ a, u \geq 0, b_1, ..., b_u \geq 1}} (-1)^u [(0, \Lambda^a(N) \otimes_\Delta \Lambda^{b_1}(\Delta \bot N) \otimes_\Delta \cdots \otimes_\Delta \Lambda^{b_u}(\Delta \bot N),0)]
\end{align*}
in $K_0[F_{nk}]$ which in turn is equal to
\[\lambda^k \Big([(0,N,0)] - [(0,\Delta \bot N, 0)]\Big)\]
again by \cite[Equation 2.3]{GraysonExterior}. This proves the second statement and hence finishes the proof of \cref{prop: compatibility of power operations with exact sequence}.
\end{proof}

\begin{remark}
Rather than using the somewhat combinatorially involved \cref{lem: computing lambda^k}, one could alternatively argue as follows in the first half of the proof above.\\
Using \cite[Equation 2.3]{GraysonExterior}, we obtain 
\begin{align*}
        \lambda^k([X] &- [\Delta \bot X]) = \sum_{\substack{a+b_1+\cdots+b_u=k \\ a, u \geq 0, b_1, ..., b_u \geq 1}} (-1)^u \lambda^a [X] \lambda^{b_1} [\Delta \bot X] \cdots \lambda^{b_u} [\Delta \bot X]\\
        =& \sum_{\substack{a+b_1+\cdots+b_u=k \\ a, u \geq 0, b_1, ..., b_u \geq 1}} (-1)^u [\Lambda^a(X) \otimes_\Delta \Lambda^{b_1}(\Delta \bot X) \otimes_\Delta \cdots \otimes_\Delta \Lambda^{b_u}(\Delta \bot X)]
\end{align*}
in $K_0B[F_{nk}]$. The image in $K_0qC\cM_{nk}$ of this element's class in $K_0[F_{nk}]$ is
\begin{align*}
&\sum_{\substack{a+b_1+\cdots+b_u=k \\ a, u \geq 0, b_1, ..., b_u \geq 1}} (-1)^u [\Lambda^a(A) \otimes_\Delta \Lambda^{b_1}(B) \otimes_\Delta \cdots \otimes_\Delta \Lambda^{b_u}(B)]\\
&\hspace*{3em} - \sum_{\substack{a+b_1+\cdots+b_u=k \\ a, u \geq 0, b_1, ..., b_u \geq 1}} (-1)^u [\Lambda^a(B) \otimes_\Delta \Lambda^{b_1}(B) \otimes_\Delta \cdots \otimes_\Delta \Lambda^{b_u}(B)]\\
&\hspace*{1em} = \lambda^k ([A] - [B]) - \lambda^k([B] - [B]) = \lambda^k ([A]- [B]) ,
\end{align*}
(again by \cite[Equation 2.3]{GraysonExterior}), as was to be shown. 
\end{remark} 

\begin{remark}\label{rem: power operations on Bass K_1}
In this remark we look at power operations on Bass's~$K_1^\mathrm{B}$. \\
First, recall from \cref{rem: products on Bass's K_1} that we have an isomorphism 
\[K_0\textrm{Aut}(\cM)/\tilde{I} \cong K_0\cM \oplus K_1^\mathrm{B}\cM\] 
for any exact category $\cM$. Given an assembly of power operations on $\cM_n$, $n \ge 0$, it is not difficult to verify (cf.~\cite[Lemma~1.2]{Koe00}) that the assignment $(P, \alpha) \mapsto (\Lambda^k(P), \Lambda^k(\alpha))$ defines a map $\lambda^k \colon K_0\textrm{Aut}(\cM_n) \rightarrow K_0\textrm{Aut}(\cM_{nk})$ and then induces a map $\lambda^k\colon K_0 \cM_n \oplus K_1^\mathrm{B} \cM_{n} \rightarrow K_0\cM_{nk} \oplus K_1^\mathrm{B} \cM_{nk}$ which becomes the usual power operation $\lambda^k$ when restricted to $K_0(\cM_n)$ and becomes a homomorphism when restricted to $K_1^\mathrm{B}(\cM_n)$. In the next paragraph, we sketch a proof of the fact that the canonical map $\varphi \colon K_1^\mathrm{B}\cM_n \rightarrow K_1\cM_n$ (see (\ref{equ: iso between Bass and Grayson K1})) is compatible with the operations~$\lambda^k$. \\
First, beware, although we have $\lambda^k([N]) = [\Lambda^k(N)]$ in $K_1(\cM_{nk})$ for any $N \in B^\mathrm{q} \cN_n$, we don't have $\lambda^k([(P, \alpha)]) = [({\Lambda}^k P, {\Lambda}^k \alpha)]$ in $K_1^\mathrm{B}\cM_{nk}$ for $(P, \alpha)$ in $\Aut(\M_n)$. We rather have (using \cite[Equation 2.3]{GraysonExterior}): 
\begin{align*}
\lambda^k&([(P, \alpha)]) = [\lambda^k((P,\alpha)-(P,\mathrm{id}))]\\
&= \sum_{\substack{a+b_1+\cdots+b_u=k \\  u \geq 0, a, b_1, ..., b_u \geq 1}} (-1)^u [\Lambda^a(P,\alpha) \otimes \Lambda^{b_1}(P,\mathrm{id}) \otimes \ldots \otimes \Lambda^{b_u}(P,\mathrm{id})]
\end{align*}
where the terms for $a=0$ have been omitted because they vanish. The homomorphism $\varphi$ sends this alternating sum to the corresponding alternating sum of the classes of the binary automorphisms 
\[(\Lambda^a(P) \otimes \Lambda^{b_1}(P) \otimes \ldots \otimes \Lambda^{b_u}(P), \Lambda^a(\alpha) \otimes \Lambda^{b_1}(\mathrm{id}) \otimes \ldots \otimes \Lambda^{b_u}(\mathrm{id}), \mathrm{id}).\] 
For a fixed $u$, the sum of such terms is exactly the $(u+1)$th cross-effect functor of~$\Lambda^k$; more precisely, the alternating sum arrived at above is equal to 
\[\sum_{u \ge 1} (-1)^{u-1}\left[\left(\mathrm{cr}_u(\Lambda^k)(P,\ldots, P), \mathrm{cr}_u(\Lambda^k)(\alpha, \mathrm{id}, \ldots, \mathrm{id}), \mathrm{id}\right)\right].\] 
On the other hand, by \cite[Lemma 2.2]{Koe01} (also see \cite[Section 3]{dBT}), the binary complex $\Lambda^k(P,\alpha,\mathrm{id})$ of length $k$ can be filtered by subcomplexes such that the successive quotients are the shifted binary automorphisms 
$\left(\mathrm{cr}_u(\Lambda^k)(P,\ldots, P), \mathrm{cr}_u(\Lambda^k)(\alpha, \mathrm{id}, \ldots, \mathrm{id}), \mathrm{id}\right)[u-1]$, $u \ge 1$. It remains to observe that shifting a binary complex by $u-1$ yields the sign $(-1)^{u-1}$ in~$K_1$ by \cite[Corollary 8.7]{Grayson2016}.
\end{remark} 

\begin{remark}\label{rem: power operations on Bass relative K_1}
In this remark we look at power operations on Bass's relative~$K_0$-group. To this end, we assume the situation introduced before \cref{def: power operations for triples}. 
\begin{enumerate}[(a)]
\item For $X= (P, \alpha, Q) \in \mathrm{co}(F_n)$ define $\Lambda^k(X):= (\Lambda^k(P), \Lambda^k(\alpha), \Lambda^k(Q))$. Then, as in \cref{rem: power operations on Bass K_1}, see also \cite[Lemma~1.3]{Koe00}, the assignment $X \mapsto \Lambda^k(X)$ defines a map $\lambda^k \colon K_0\mathrm{co}(F_n) \rightarrow K_0\mathrm{co}(F_{nk})$ and then induces a map 
\[\lambda^k \colon K_0^\mathrm{B}[F_n] \rightarrow K_0^\mathrm{B}[F_{nk}]\] 
via the isomorphism (\ref{equ: isomorphism for Bass relative K_0}). In contrast to \cref{rem: power operations on Bass K_1}, this operation $\lambda^k$ is not a homomorphism in general.  
\item Similarly to \cref{lem: computing lambda^k}, for $X= (P, \alpha, Q) \in \mathrm{co}(F_n)$, we can again describe $\lambda^k([X]) \in K_0^B[F_{nk}]$ in terms of the class of a single object of $\mathrm{co}(F_{nk})$ in the following way. Let $\tau X:= (Q, \alpha^{-1}, P)$, $\bot X:= Q$ and $\Delta \bot X := (Q, \mathrm{id}_{FQ}, Q)$. Then we have: 
\begin{align*}
\hspace*{2em} \lambda^k ([X])& = \left[\bigoplus_{\substack{a+b_1+\cdots+b_u=k \\ u \geq 0, a, b_1, ..., b_u \geq 1}} \Lambda^a(\tau^u X) \otimes \Lambda^{b_1}(\Delta \bot X) \otimes \cdots \otimes \Lambda^{b_u}(\Delta \bot X)\right]\\
& = [(\Lambda_+^k(P, Q), \widehat{\Lambda}^k \alpha, \Lambda_-^k(P, Q))]
\end{align*}
 where $\Lambda_+^k(P, Q)$ and $\Lambda_-^k(P, Q)$ are defined similarly as in \cref{lem: computing lambda^k} and $\widehat{\Lambda}^k\alpha$ is defined as follows: 
 \[\widehat{\Lambda}^k \alpha = \bigoplus_{\substack{a+b_1+\cdots+b_u=k \\ u \geq 0, a, b_1, ..., b_u \geq 1}} \Lambda^a(\alpha^{(-1)^u}) \otimes \Lambda^{b_1}(1) \otimes \cdots \otimes \Lambda^{b_u}(1).\]
 \item By part (b), the map $K_0^\mathrm{B}[F_{nk}] \rightarrow K_0\cM_{nk}$, $[(P,\alpha, Q)] \mapsto [P]-[Q]$, sends $\lambda^k([X])$ to $[\Lambda^k_+(P,Q)] - [\Lambda^k_-(P,Q)] = \Lambda^k([P]-[Q])$ and hence commutes with~$\lambda^k$.  The incoming map $K_1^\mathrm{B} \cN_n \rightarrow K_0^\mathrm{B}[F_{n}]$ in Bass's exact sequence \cite[Theorem~VII.5.3]{Bas68} commutes with $\lambda^k$ as well (assuming $\cM_n$ and $\cN_n$ are split exact and the $F_n$ are cofinal); to see this we just observe that this map is the restriction of the analogously defined map $K_0 \mathrm{Aut}(\cN_n)/\tilde{I} \rightarrow K_0\mathrm{co}(F_n)/\tilde{J}$ which obviously commutes with $\lambda^k$. 
 \item Finally, the homomorphism $\Phi \colon K_0^\mathrm{B}[F_n] \rightarrow K_0[F_n]$ from \cref{thm:BassGraysonIso} also commutes with $\lambda^k$; this follows from the very definitions.
 \end{enumerate}
\end{remark}

\section{Lambda Ring Axioms for Relative \texorpdfstring{$K_0$}{K0}}

The main result of this section is that the products and exterior power operations introduced in the previous two sections turn Grayson's relative $K_0$-group into a $\lambda$-ring (without unity). We begin by proving that the product axiom in the classical definition of a $\lambda$-ring follows from the other axioms. To prove the remaining axiom (concerning composition of exterior power operations), we follow the approach from \cite{HKT17} which relies on the fact that the Grothendieck group of the exact category consisting of polynomial functors over $\ZZ$ is the free $\lambda$-ring in one variable. 

Recall a {\it pre-$\lambda$-ring} is a commutative ring $K$ together with maps $\lambda^k$, $k \ge 1$, from $K$ to itself such that $\lambda^1(x) = x$ for all $x \in K$ and such that
\begin{equation}\label{eq: sum axiom}
\lambda^k(x+y) =  \sum_{i=0}^{k} \lambda^{k-i}(x) \lambda^i(y)
\end{equation}
for all $ x,y \in K$ and $k \ge 1$. As we want to allow our (pre-)$\lambda$-rings to be without unity, our definition does not include a $\lambda^0$ and any term such as $\lambda^0(x)\lambda^k(y)$ is to mean just $\lambda^k(y)$ here and below.  The pre-$\lambda$-ring $K$ is called a {\it $\lambda$-ring} if moreover 
 \begin{align}
        \lambda^k(xy) &= P_k(\lambda^1(x), \lambda^2(x), ..., \lambda^k(x); \lambda^1(y), \lambda^2(y), ..., \lambda^k(y)), \label{eq: product axiom}\\
        \lambda^k (\lambda^l(x)) &= P_{k, l}(\lambda^1(x), \lambda^2(x), ..., \lambda^{kl}(x))
        \label{eq: composition axiom}
    \end{align}
for all $x, y \in K$ and $k, l \ge 1$ where $P_k, P_{k, l}$ are certain universal integral polynomials (see \cite[I, \S 1]{FultonLang}). We call (\ref{eq: product axiom}) and (\ref{eq: composition axiom}) the product and composition axiom, respectively. 

In the context below, the product axiom could be proved by working out a bivariate version of the proof of the composition axiom (see \cref{thm: K_0B[F] is a lambda ring} below) or, similarly to the proof of \cite[Theorem~2.2]{Koe00}, by using the characteristic free Cauchy decomposition as developed by Akin, Buchsbaum and Weyman in \cite{ABW}. We however use a shortcut provided by the following somewhat surprising proposition. 

\begin{proposition}\label{prop: omitting product axiom}
Let $K$ be a pre-$\lambda$-ring which satisfies the composition axiom. Then $K$ is a $\lambda$-ring. 
\end{proposition}

\begin{proof}
We define polynomials $Q_k \in \ZZ[X_1, \ldots, X_{2k}; Y_1, \ldots, Y_{2k}]$, $k\ge 1$, by induction on $k$ as follows: let $Q_1 := X_1 Y_1$ and
\begin{align*}
Q_k := P_{k,2}&\left(X_1+Y_1, X_2 + X_1Y_1 + Y_2, \ldots, \sum_{i=0}^{2k}X_{2k-i}Y_i \right) \\
&- \sum_{i=1}^k Q_{k-i}  \sum_{j=0}^{i} P_{i-j,2}(X_1, \ldots, X_{2(i-j)})P_{j,2}(Y_1, \ldots, Y_{2j}).
\end{align*}
We now show that 
\begin{align}\label{eq: product axiom with Q}
\lambda^k(xy) = Q_k(\lambda^1(x), \ldots, \lambda^{2k}(x); \lambda^1(y), \ldots, \lambda^{2k}(y))
\end{align}
for all $x,y \in K$. We proceed by induction on $k$. For $k=1$ this is obvious. Writing $\lambda^ix$ for $\lambda^i(x)$, etc., we obtain for $k \ge 2$:
\begin{align*}
 \lambda^k&\lambda^2(x+y) \overset{(\ref{eq: composition axiom})}{=} P_{k,2}\left(\lambda^1(x+y), \lambda^2(x+y), \ldots, \lambda^{2k}(x+y)\right) \\
& \overset{(\ref{eq: sum axiom})}{=} P_{k,2}\left(x+y, \lambda^2x+xy+\lambda^2y, \ldots, \sum_{i=0}^{2k}\lambda^{2k-i}x\lambda^iy \right).
\end{align*} 
On the other hand, we have:
\begin{align*}
\lambda^k\lambda^2(x+y) & \overset{(\ref{eq: sum axiom})}{=} \lambda^k(\lambda^2x + xy + \lambda^2y)\\
& \overset{(\ref{eq: sum axiom})}{=} \lambda^k(xy) + \sum_{i=1}^{k} \lambda^{k-i}(xy) \lambda^i(\lambda^2x + \lambda^2y) \\
& \overset{(\ref{eq: sum axiom})}{=} \lambda^k(xy)  + \sum_{i=1}^k \lambda^{k-i}(xy) \sum_{j=0}^{i} \lambda^{i-j}\lambda^2x \; \lambda^j\lambda^2y.
\end{align*}
Now, for $i = 1, \ldots, k-1$, the inductive hypothesis yields $\lambda^{k-i}(xy)= Q_{k-i}(\lambda^1x, \ldots, \lambda^{2(k-i)}x;\lambda^1y, \ldots, \lambda^{2(k-i)}y)$. Furthermore, for any $j \ge 1$,  we have $\lambda^j\lambda^2x\overset{(\ref{eq: composition axiom})}{=}P_{j,2}(\lambda^1x, \ldots, \lambda^{2j}x)$ and $\lambda^j\lambda^2y\overset{(\ref{eq: composition axiom})}{=}P_{j,2}(\lambda^1y, \ldots, \lambda^{2j}y)$. The desired formula (\ref{eq: product axiom with Q}) finally follows from all this.

To finish the proof of \cref{prop: omitting product axiom} we observe that the polynomial ring $U:=\ZZ[X_1, X_2, \ldots; Y_1, Y_2, \ldots]$  in the infinitely many variables $X_1, X_2, \ldots$ and $Y_1, Y_2, \ldots$ can be equipped with a unique $\lambda$-ring structure such that $\lambda^k(X_1) = X_k$ and $\lambda^k(Y_1) = Y_k$. It is called the free-$\lambda$-ring in $2$ variables (but this universal property is not used here) and can be constructed similarly to the free $\lambda$-ring in $1$ variable which has been introduced in \cite[I, \S 2]{AT69}. In particular, we have $\lambda^k(X_1Y_1) = P_k$ by definition of a $\lambda$-ring and $\lambda^k(X_1Y_1) = Q_k$ by (\ref{eq: product axiom with Q}) applied to $X_1, Y_1 \in U$; thus $Q_k= P_k$ in $U$.  

Together with (\ref{eq: product axiom with Q}), this proves the desired product axiom formula (\ref{eq: product axiom}) for $K$. 
\end{proof}

Now, let $f \colon X\rightarrow Y$ be a morphism between quasi-compact schemes and let $F := f^* \colon \cP(Y) \rightarrow \cP(X)$ denote the induced pull-back functor between the associated categories of locally free $\cO_X$- and $\cO_Y$-modules of finite rank. By \cref{lem: power operations on B[F]}, we have exterior power operations $\lambda^k \colon K_0B[F] \rightarrow K_0B[F]$ which turn $K_0B[F]$ into a pre-$\lambda$-ring.

\begin{theorem}\label{thm: K_0B[F] is a lambda ring} 
The pre-$\lambda$-ring $K_0B[F]$ is a $\lambda$-ring. 
\end{theorem}

\begin{proof}
Let $k,l \ge 1$. According to \cref{prop: omitting product axiom}, it suffices to show that 
\[ \lambda^k(\lambda^l(x)) = P_{k, l}(\lambda^1(x), \lambda^2(x), ..., \lambda^{kl}(x))\]
for all $x \in K_0B[F]$. To this end, we follow the approach used in \cite[Section~8]{HKT17}. We first note that it suffices to prove this formula when $x$ is the class of a triple $V \in B[F]$. Due to a standard argument, the category $\mathrm{End}(B[F])$ of endo-functors on $B[F]$ is an exact category with tensor product and exterior powers $\Lambda^k$  given by $H \mapsto \Lambda^k \circ H$. We will show that the stronger identity
    \[ [\Lambda^k \circ \Lambda^l] = P_{k, l}([\Lambda^1], ..., [\Lambda^{kl}]) \]
    holds in $K_0 \mathrm{End}(B[F])$. The desired identity in $K_0B[F]$ then follows by applying the homomorphism
    \[ K_0 \mathrm{End}(B[F]) \to K_0 B[F] \text{ given by } [H] \mapsto [H(V)]. \]
    By \cite[Corollary~8.6]{HKT17}, the stronger identity holds in the Grothendieck group $K_0 \mathrm{Pol}^0_{< \infty}(\mathbb{Z})$ of the category of \emph{polynomial} endo-functors on $\cP(\ZZ) := \cP(\mathrm{Spec}(\ZZ))$ of finite degree which fix $0$. It therefore suffices to construct an exact functor
    \begin{equation}\label{eq: functor from Pol}
    \mathrm{Pol}^0_{<\infty}(\ZZ) \rightarrow \mathrm{End}\left(B[F]\right)
    \end{equation}
    which respects composition, tensor products and exterior powers and sends the identity functor to the identity functor. To this end, recall from \cite[Section~8A]{HKT17} that a polynomial functor $H$ over a scheme $Z$ assigns a $H(\cF)\in \cP(Z)$ to any object~$\cF \in \cP(Z)$ and a morphism $H(\alpha)\colon H(\cF)_T \rightarrow H(\cG)_T$ to any morphism $\alpha \colon \cF_T \rightarrow \cG_T$ for any $\cF, \cG \in \cP(Z)$ and for any $Z$-scheme~$T$; the latter assignments must commute with pullback $\varphi^*\colon \mathrm{Hom}(\cF_T, \cG_T) \rightarrow \mathrm{Hom}(\cF_{T'}, \cG_{T'})$ for any morphism $\varphi \colon T' \rightarrow T$ of $Z$-schemes, but need not be linear. Furthermore, for any morphism $g \colon Z' \rightarrow Z$ of schemes, there is a base change functor $\mathrm{Pol}_{< \infty}^0(Z) \rightarrow \mathrm{Pol}_{< \infty}^0(Z')$, see the construction in \emph{loc.\ cit.} The base change of any $H \in \mathrm{Pol}_{<\infty}^0(Z)$ will be denoted $H$ again and satisfies $H(g^*(\cF)) \cong g^*(H(\cF))$ by construction.  Now, given $H \in \mathrm{Pol}_{<\infty}^0(\ZZ)$ and $V =(M.,N.,u) \in B[F]$, we base-change $H$ from $\mathrm{Spec}(\ZZ)$ to $X$ and~$Y$ and define $H(V) := (N H\Gamma M., N H\Gamma N., N H\Gamma u)$, similarly to \cref{Sec: 2} and \cref{Sec: 3}. Note that $N H\Gamma(u)$ is a quasi-isomorphism again; this is proved in \cite[Proposition 1.9]{KZ25} when $H$ is an exterior power functor but the proof there works for any functor $H$ which fixes $0$. By construction, the resulting functor (\ref{eq: functor from Pol}) satisfies the properties listed above. In particular, this finishes the proof.     
\end{proof}

We remark that we similarly obtain that $K_0C[F]$ is a $\lambda$-ring. Furthermore, it immediately follows that $K_0pB[F]$ and $K_0pC[F]$ are $\lambda$-rings, and we obtain:

\begin{corollary}\label{cor: K_0[F] is a lambda ring}
The products and exterior power operations introduced in \cref{def: products on K_0[F]} and \cref{def: exterior powers on relative K_0} turn the relative Grothendieck group $K_0[F]$ into a $\lambda$-ring (without unity). 
\end{corollary}

\begin{proof} By \cref{thm: K_0B[F] is a lambda ring}, the axioms (\ref{eq: sum axiom}), (\ref{eq: product axiom}) and (\ref{eq: composition axiom}) hold in the subring $\mathrm{ker}(\bot \colon K_0pB[F] \rightarrow K_0pC[F])$  of $K_0B[F]$. As the isomorphism (\ref{Equ: identification}) is by construction compatible with products and exterior power operations, they then also hold in $K_0[F]$.
\end{proof}

\section{Exterior Power Operations on Higher Relative \texorpdfstring{$K$}{K}-Groups}\label{sec: higher relative K-groups}

In this section, we first recall Grayson's combinatorial definitions of higher $K$-groups \cite{Grayson2012} and of higher relative $K$-groups \cite{Grayson2016}. The recursive nature of these definitions then allows us to prove the main result of this paper, namely that higher relative $K$-groups can be equipped with natural exterior power operations which satisfy the axioms of a $\lambda$-ring and which are compatible with the incoming and outgoing maps in the long exact sequence of $K$-groups. 

Let $\cM$ be an exact category and $n \ge 0$. By \cite[Sections~2 and~3]{Grayson2012}, the categories $C\cM$, $B\cM$, $C^\mathrm{q}\cM$ and $B^\mathrm{q}\cM$ are again exact categories (which support long exact sequences) and so we can iterate these constructions. As in \cite[Section~7]{Grayson2012}, we form the $n$-cube $\Omega^n \cM$ of exact categories as follows. At the corner of $\Omega^n \cM$ corresponding to $(i_1, \ldots, i_n) \in \{0,1\}^n$  we put the exact category $G_{i_1} \ldots G_{i_n}\cM$ where $G_i := C^\mathrm{q}$ if $i=0$ and $G_i := B^\mathrm{q}$ if $i=1$. Furthermore, we put the obvious diagonal functor $\Delta$ at each edge of $\Omega^n \cM$. For example, the $0$-cube~$\Omega^0 \cM$ is just $\cM$, the $1$-cube $\Omega^1\cM$ is $C^\mathrm{q}\cM \overset{\Delta}{\longrightarrow} B^\mathrm{q}\cM$ and the $2$-cube~$\Omega^2 \cM$ is the square 
\[\begin{tikzcd}
            C^{\mathrm{q}} C^{\mathrm{q}} \cM \arrow[r, "\Delta"] \arrow[d, "\Delta"]
            & C^{\mathrm{q}} B^{\mathrm{q}} \cM \arrow[d, "\Delta"]\\
            B^{\mathrm{q}} C^{\mathrm{q}} \cM \arrow[r, "\Delta"]
            & B^{\mathrm{q}} B^{\mathrm{q}} \cM.
    \end{tikzcd}\]
We now consider all those functors $\Delta$ in $\Omega^n\cM$ which arrive at $(B^\mathrm{q})^n \cM$ and define $\cK_0\Omega^n \cM$ to be the factor group of $K_0(B^\mathrm{q})^n\cM$ 
modulo the subgroup generated by the images of all the homomorphisms $K_0 \Delta$ induced by these functors $\Delta$. In other words, $\cK_0\Omega^n \cM$ is obtained by taking iterated cokernels (in any order) in the cube of abelian groups obtained from $\Omega^n\cM$ by applying~$K_0$. By \cite[Corollary~7.2]{Grayson2012}  we then have a canonical isomorphism
\begin{equation}
K_n \cM \cong \cK_0\Omega^n\cM
\end{equation}
between Quillen's $n$th higher $K$-group $K_n\cM$  and $\cK_0\Omega^n\cM$. We will take this as the definition of $K_n \cM$.

Now, let $F \colon \cM \rightarrow \cN$ be an exact functor between exact categories. We form the $(n+1)$-cube 
\[\Omega^n[F] := \left(\Omega^n \cM \overset{F}{\longrightarrow} \Omega^n \cN\right).\] For example, for $n=2$, this looks as follows: 
\begin{equation*}
        \begin{tikzcd}
             C^{\mathrm{q}}C^{\mathrm{q}}\mathcal{M} \arrow[rr, "C^{\mathrm{q}}C^{\mathrm{q}}F"] \arrow[dr, "\Delta"] \arrow[dd, "\Delta"]
            & & C^{\mathrm{q}}C^{\mathrm{q}}\mathcal{N}  \arrow[dd, "\Delta", near end] \arrow[dr, "\Delta"] & \\
            &  C^{\mathrm{q}}B^{\mathrm{q}}\mathcal{M} \arrow[rr, crossing over, "C^{\mathrm{q}}B^{\mathrm{q}}F", near start]
            & & C^{\mathrm{q}}B^{\mathrm{q}}\mathcal{N}  \arrow[dd, "\Delta"]\\
             B^{\mathrm{q}}C^{\mathrm{q}}\mathcal{M} \arrow[rr, "B^{\mathrm{q}}C^{\mathrm{q}}F", near start] \arrow[dr, "\Delta"] & & B^{\mathrm{q}}C^{\mathrm{q}}\mathcal{N}  \arrow[dr, "\Delta"] &\\
            & B^{\mathrm{q}}B^{\mathrm{q}}\mathcal{M} \arrow[rr, "B^{\mathrm{q}}B^{\mathrm{q}}F"] \arrow[from=uu, crossing over, "\Delta", near end] & & B^{\mathrm{q}}B^{\mathrm{q}}\mathcal{N} 
        \end{tikzcd}
\end{equation*}

\begin{definition}[{\cite[Section~2]{Grayson2016}}]\label{def: higher relative K-groups}
The $n$th higher relative $K$-group $K_n[F]$ is defined as the iterated cokernel of the $n$-cube $K_0\Omega^n[F]$ of abelian groups obtained from $\Omega^n[F]$ by forming the relative $K_0$-group of each of those functors in~$\Omega^n[F]$ which are induced by $F$. 
\end{definition} 
\noindent For example, $K_2[F]$ is defined as the iterated cokernel of the square 
    \begin{equation*}
        \begin{tikzcd}
            K_0 [C^\mathrm{q} C^\mathrm{q} F] \arrow[r, "\Delta"] \arrow[d, "\Delta"]
            & K_0 [C^\mathrm{q} B^\mathrm{q} F] \arrow[d, "\Delta"]\\
            K_0 [B^\mathrm{q} C^\mathrm{q} F] \arrow[r, "\Delta"]
            & K_0 [B^\mathrm{q} B^\mathrm{q} F].
        \end{tikzcd}
    \end{equation*}

\begin{remark}\label{rem: Higher Relative K group justification}
The following two facts justify calling $K_n[F]$ the $n$th higher relative $K$-group.
\begin{enumerate}[(a)]
\item By \cite[Corollary~2.3]{Grayson2016}, there is a purely algebraically defined long exact sequence
    \begin{equation*}
        \begin{tikzcd}
            \cdots \arrow[r]  
            & K_{n+1} \M \arrow[r] \ar[draw=none]{d}[name=X, anchor=center]{}
            & K_{n+1} \N \ar[rounded corners,
                    to path={ -- ([xshift=2ex]\tikztostart.east)
                              |- (X.center) \tikztonodes
                              -| ([xshift=-2ex]\tikztotarget.west)
                              -- (\tikztotarget)}]{dll}[at end]{} \\
            K_n[F] \arrow[r] 
            & K_n \M \arrow[r] \ar[draw=none]{d}[name=Y, anchor=center]{}
            & K_n \N \ar[rounded corners,
                    to path={ -- ([xshift=2ex]\tikztostart.east)
                              |- (Y.center) \tikztonodes
                              -| ([xshift=-2ex]\tikztotarget.west)
                              -- (\tikztotarget)}]{dll}[at end]{} \\
            \cdots \arrow[r]
            & K_0 \M \arrow[r]
            & K_0 \N.
         \end{tikzcd}
    \end{equation*}
\item By \cite[Theorem~4.3]{Tu24}, there is a canonical isomorphism
\[K_n[F] \cong \pi_n\mathrm{hofib}(KF \colon K\cM \rightarrow K\cN)\]
between $K_n[F]$ and the $n$th homotopy group of the homotopy fibre of the induced map $KF$ between the $K$-theory spaces $K\cM$ and  $K\cN$ of $\cM$ and $\cN$.
\end{enumerate}
\end{remark}

We can also describe $K_n[F]$ in the following ways. First, we turn the $(n+1)$-cube $\Omega^n[F]$ into another $(n+1)$-cube of exact categories by replacing every edge in $\Omega^n[F]$ of the form $G_1 \ldots G_n \cM \overset{G_1 \ldots G_n F}{\longrightarrow} G_1 \ldots G_n \cN$ with $C[G_1 \ldots G_n F] \overset{\Delta}{\longrightarrow} B[G_1 \ldots G_n F]$. We tacitly assume that the notations $C[G_1 \ldots G_n F]$ and $B[G_1 \ldots G_n F]$ here include the subcategories $pC[G_1 \ldots G_n F]$ and $pB[G_1 \ldots G_n F]$, respectively, as introduced in \cref{def: triples and relative K_0}(a) and (b).  We use the ad hoc notation $\mathrm{Cube}(F,\Delta, n)$ for this cube. Then $K_n[F]$ is the iterated cokernel of the $(n+1)$-cube $K_0p\mathrm{Cube}(F,\Delta,n)$ of abelian groups obtained from $\mathrm{Cube}(F,\Delta,n)$ by applying the functor~$K_0p$. If we additionally replace all homomorphism $K_0p\Delta$  with $K_0p\bot$, then, similarly to (\ref{Equ: identification}), the iterated kernel of the resulting $(n+1)$-cube $K_0p\mathrm{Cube}(F,\bot,n)$ is  canonically isomorphic to $K_n[F]$. For example, the iterated kernel of the cube of abelian groups obtained by applying $K_0p$ to the cube 
\begin{equation*}
        \begin{tikzcd}[arrows=<-]
             C[C^{\mathrm{q}}C^{\mathrm{q}}F]\arrow[rr, "\bot"] \arrow[dr, "\bot"] \arrow[dd, "\bot"]
            & & B[C^{\mathrm{q}}C^{\mathrm{q}}F] \arrow[dd, "\bot", near end] \arrow[dr, "\bot"] & \\
            &  C[C^{\mathrm{q}}B^{\mathrm{q}}F] \arrow[rr, crossing over, "\bot", near start]
            & & B[C^{\mathrm{q}}B^{\mathrm{q}}F] \arrow[dd, "\bot"]\\
             C[B^{\mathrm{q}}C^{\mathrm{q}}F] \arrow[rr, "\bot", near start] \arrow[dr, "\bot"] & & B[B^{\mathrm{q}}C^{\mathrm{q}}F]  \arrow[dr, "\bot"] &\\
            & C[B^{\mathrm{q}}B^{\mathrm{q}}F]\arrow[rr, "\bot"] \arrow[from=uu, crossing over, "\bot", near end] & & B[B^{\mathrm{q}}B^{\mathrm{q}}F]
        \end{tikzcd}
\end{equation*}
is isomorphic to $K_2[F]$.

Now, as in the previous section, let $f \colon X\rightarrow Y$ be a morphism between quasi-compact schemes, let $\cM := \cP(Y)$, $\cN := \cP(X)$ and let $F := f^*$ denote the induced pull-back functor from $\cM$ to $\cN$. The standard assembly of power operations on $\cP(Y)$ and $\cP(X)$ induce an assembly of power operations on $G_1 \ldots G_n\cM$ and $G_1 \ldots G_n\cN$ for any $G_1, \ldots, G_n \in \{C^\mathrm{q},B^\mathrm{q}\}$ (see \cite[Section~1]{KZ25}) and then an assembly of power operations on $C[G_1 \ldots G_n F]$ and $B[G_1 \ldots G_n]$ (see \cref{lem: power operations on B[F]}) and finally products and exterior power operations $\lambda^k$, $k \ge 1$, on $K_0[G_1 \ldots G_n F]$ and $K_0p[G_1 \ldots G_nF]$ which turn these Grothendieck groups into pre-$\lambda$-rings. As these pre-$\lambda$-ring structures are compatible with all homomorphisms in the cube $K_0p\mathrm{Cube}(F,\bot,n)$, we also obtain a pre-$\lambda$-ring structure on the iterated kernel of this cube. 

The following definition extends \cref{def: products on K_0[F]} and \cref{def: exterior powers on relative K_0} from $n=0$ to any $n \ge 0$. 

\begin{definition}\label{def: lambda-ring structure on higher relative K-groups}
The pre-$\lambda$-ring structure on $K_n[F]$ is defined to be the one obtained by transporting the pre-$\lambda$-ring structure on the iterated kernel of $K_0p\mathrm{Cube}(F,\bot,n)$ to $K_n[F]$ via the canonical isomorphism described above. 
\end{definition}

\begin{theorem}\label{thm: main theorem}\mbox{}
\begin{enumerate}[(a)] 
\item The pre-$\lambda$-ring $K_n[F]$ is a $\lambda$-ring. 
\item If $n\ge 1$, all products on $K_n[F]$ vanish and the exterior power operations~$\lambda^k$, $k \ge 1$, on $K_n[F]$ are homomorphisms. 
\item Products and exterior power operations on $K_n[F]$ are compatible with products and exterior power operations on the higher $K$-groups of $X$ and $Y$ via the long exact sequence mentioned in \cref{rem: Higher Relative K group justification}(a). 
\end{enumerate}
\end{theorem}

\begin{proof}\mbox{}\\
(a) As in \cref{cor: K_0[F] is a lambda ring}, the first statement follows from the stronger statement that $K_0B[(B^\mathrm{q})^nF]$ is a $\lambda$-ring. To prove the latter statement we proceed as in the proof of \cref{thm: K_0B[F] is a lambda ring}. Similarly to there, this comes down to constructing an exact functor
    \begin{equation}
    \mathrm{Pol}^0_{<\infty}(\ZZ) \rightarrow \mathrm{End}\left(B[(B^\mathrm{q})^nF]\right)
    \end{equation}
    which respects composition, tensor products and exterior powers and sends the identity functor to the identity functor. This functor in turn is obtained as in the proof of \cref{thm: K_0B[F] is a lambda ring} after first inductively defining the functors $\mathrm{Pol}^0_{<\infty}(\ZZ) \rightarrow \mathrm{End}\left((B^\mathrm{q})^n\cP(Y)\right)$ and $\mathrm{Pol}^0_{<\infty}(\ZZ) \rightarrow \mathrm{End}\left((B^\mathrm{q})^n\cP(X)]\right)$ as in the proof of \cite[Theorem~8.18]{HKT17}. \\
(b) Due to axiom (\ref{eq: sum axiom}), we only need to show the first statement. By definition this means we need to show that products on the iterated kernel of the cube $K_0p\mathrm{Cube}(F,\bot,n)$ vanish. This iterated kernel in turn is a quotient of the iterated kernel of $K_0\mathrm{Cube}(F, \bot, n)$, and the latter kernel carries products similar to and hence compatible with the ones on the former. Now, products on higher $K$-groups such as $K_n B[F]$ vanish by \cite[Proposition~5.11]{HKT17}, hence also the induced products on $\mathrm{ker}(\bot \colon K_nB[F] \rightarrow K_nC[F])$. The latter is the same as the iterated kernel of $K_0\mathrm{Cube}(F, \bot, n)$ because we may think of each category $G[G_1 \ldots, G_n F]$ in $\mathrm{Cube}(F, \bot, n)$  as $G_1 \ldots G_n (G[F])$ and because forming iterated kernels does not depend on the order of iterations. Overall, we conclude that products on $K_n[F]$ vanish as claimed.   \\ 
(c) We associate the exact sequence (\ref{equ: long exact sequence}) with each of those arrows in the $(n+1)$-cube $\Omega^n[F]$ which are induced by $F$. This way we obtain the exact sequence of $n$-cubes of abelian groups 
\[K_1\Omega^n \cM \rightarrow K_1 \Omega^n \cN \rightarrow K_0\Omega^n[F] \rightarrow K_0\Omega^n\cM \rightarrow K_0 \Omega^n \cN.\]
Now we replace every homomorphism induced by $\Delta$ in every cube of this sequence with the homomorphism induced by $\bot$ in the opposite direction. Similar to the proof of part (b), and as explained in \cite[Lemma~2.2 and proof of Corollary 2.3]{Grayson2016}, the sequence obtained by taking iterated kernels of each cube in the resulting exact sequence of $n$-cubes is then the same as the part 
\begin{equation}\label{eq: 5-term exact sequence}
K_{n+1}\cM \rightarrow K_{n+1}\cN \rightarrow K_n[F] \rightarrow K_n\cM \rightarrow K_n\cN
\end{equation}
of the sequence in \cref{rem: Higher Relative K group justification}(a). In particular, the exact sequence (\ref{eq: 5-term exact sequence}) is a subsequence of the sequence 
\[K_1(B^\mathrm{q})^n \cM \rightarrow K_1 (B^\mathrm{q})^n \cN \rightarrow K_0 [(B^\mathrm{q})^n F] \rightarrow K_0 (B^\mathrm{q})^n\cM \rightarrow K_0 (B^\mathrm{q})^n \cN\]
and, by construction, all the products and exterior power operations are compatible with this inclusion. Now part (b) follows from \cref{prop: compatibility of products} and \cref{prop: compatibility of power operations with exact sequence}.
\end{proof}


\begin{thebibliography}{[GSVW92]}

\bibitem[ABW82]{ABW} D.\ Akin, D.\ A.\ Buchsbaum and J.\ Weyman, \emph{Schur functors and Schur complexes},
Adv.\ Math.~\textbf{44} (1982), 207--278.

\bibitem[AT69]{AT69} M.\ F.\ Atiyah and D.\ O.\ Tall, \emph{Group representations, $\lambda$-rings and the $J$-homomorphism}, Topology~\textbf{8} (1969), 253–-297.

\bibitem[Bas68]{Bas68}
Hyman Bass, \emph{Algebraic K-theory}, W.~A.~Benjamin, Inc., New York--Amsterdam, 1968.

\bibitem[BF01]{BF01} D.\ Burns and M.\ Flach,
\emph{Tamagawa numbers for motives with (non-com\-mu\-ta\-tive) coefficients},
Doc.\ Math.~\textbf{6} (2001), 501–-570.

\bibitem[dBT15]{dBT}
Niels uit de Bos and Lenny Taelman, \emph{Non-additive functors and Euler characteristics}, Homology Homotopy Appl.~{\bf 17} (2015), no.~2, 25--35. 

\bibitem[DP61]{DP61} Albrecht Dold and Dieter Puppe, \emph{Homologie nicht-additiver Funktoren. Anwendungen},  Ann.\ Inst.\ Fourier (Grenoble)~\textbf{11} (1961), 201–-312.
    
\bibitem[FL85]{FultonLang} William Fulton and Serge Lang,
\textit{Riemann-{R}och algebra},
Grundlehren der Ma\-the\-matischen Wissenschaften [Fundamental Principles of Mathematical
Sciences], vol.~277, Springer-Verlag, New York, 1985.
    
\bibitem[Gra89]{GraysonExterior}
Daniel~R.\ Grayson, \emph{Exterior power operations on higher K-theory}, K-Theory~\textbf{3} (1989), no.~3, 247–-260.

\bibitem[Gra12]{Grayson2012}
\bysame, \emph{Algebraic {$K$}-theory via binary complexes}, J.\ Amer.\ Math.\ Soc.~\textbf{25} (2012), no.~4, 1149--1167.

\bibitem[Gra16]{Grayson2016}
\bysame, \emph{Relative algebraic K-theory by elementary means} [online] (2016). Available from \url{https://arxiv.org/abs/1310.8644}.

\bibitem[Har15]{HarPhD} Tom Harris, \emph{Binary complexes and algebraic K-theory}, PhD thesis, Southamp\-ton,
2015.

\bibitem[HKT17]{HKT17}
Tom Harris, Bernhard K\"ock, and Lenny Taelman, \emph{Exterior power operations
on higher {$K$}-groups via binary complexes}, Ann.\ K-Theory~\textbf{2} (2017), no.~3, 409--449.
  
\bibitem[Hel65]{Hel65}
Alex Heller, {\em Some exact sequences in algebraic K-theory}, Topology~{\bf 3} (1965). 389--408.

\bibitem[Iwa19]{Iwa19}
Ryomei Iwasa, {\em Relative $K_0$ and relative cycle class map}, J.\ Pure Appl.\ Algebra~{\bf 223} (2019), no.~1, 48--71. 

\bibitem[K\"o00]{Koe00}
Bernhard K\"ock, \emph{Symmetric powers of Galois modules on Dedekind schemes}, Compos.\ Math.~{\bf 124} (2000), no.~2, 195--217.

\bibitem[K\"o01]{Koe01}
\bysame, \emph{Computing the homology of Koszul complexes}, Trans.\ Amer.\ Math.\ Soc.~{\bf 353} (2001), no.~8, 3115--3147. 

\bibitem[KKW20]{KKW}
Daniel Kasprowski, Bernhard K\"ock, and Christoph Winges, \emph{$K_1$-groups via binary complexes of fixed length}, Homology Homotopy Appl.~{\bf 22} (2020), no.~1, 203--213. 

\bibitem[KZ25]{KZ25}
Bernhard K\"ock and Ferdinando Zanchetta. \emph{Comparison of exterior
power operations on higher K-theory of schemes}, Math.\ Z.~\textbf{309} (2025), no.~4, article~63. 

\bibitem[Lur17]{HigherAlg} Jacob Lurie, \emph{Higher Algebra} [online] (2017).
Available from: \url{https://www.math.ias.edu/~lurie/papers/HA.pdf}

\bibitem[Sou85]{Sou85} Christophe Soul\'e, 
\emph{Op\'erations en K-th\'eorie alg\'ebrique}, Canad.~J.~Math.~\textbf{37} (1985), no.~3, 488–-550.

\bibitem[TT90]{ThomTro}
R.\ W.\ Thomason and Thomas Trobaugh, \emph{Higher algebraic K-theory of schemes and of derived categories}, \emph{The Grothendieck Festschrift}, Vol.~III, 247–-435, Progr.\ Math., 88, Birkhäuser Boston, Boston, MA, 1990.

\bibitem[Tu25]{Tu24}
Jane Turner, \emph{Combinatorial relative algebraic K-theory} [online] (2025). Available from: \url{https://arxiv.org/abs/2506.09905}

\bibitem[Tu26]{Tu26} \bysame, \emph{Combinatorial relative algebraic K-theory and exterior power operations}, PhD thesis, Southampton, 2026.

\end{thebibliography}
\end{document}